\pdfoutput=1

\documentclass[10pt]{article}
\usepackage{amsmath,amssymb,amsthm}
\usepackage{mathrsfs}
\usepackage{subfig}
\usepackage{graphicx}
\usepackage{algorithm}
\usepackage{algpseudocode}
\usepackage{color}
\usepackage{textcomp}
\usepackage[table]{xcolor} 
\usepackage{tabularx}
\usepackage{upgreek}
\usepackage{fullpage}
\usepackage{url}

\newtheorem{theorem}{Theorem}[section]
\newtheorem{lemma}[theorem]{Lemma}
\newtheorem{proposition}[theorem]{Proposition}
\newtheorem{corollary}[theorem]{Corollary}

\theoremstyle{definition}

\newtheorem{remark}[theorem]{Remark}

\newcommand{\nl}{\nonumber\\}

\makeatletter
\newcommand*{\rom}[1]{\expandafter\@slowromancap\romannumeral #1@}
\newcommand{\dmn}{{\Omega}}

\newcommand{\Hspace}[1]{H^1(#1,y^a)}

\newcommand{\cil}{{{\cal C}}}
\newcommand{\cilt}{{{\cal C}_{T}}}
\newcommand{\Hdir}[1]{\mathring{H}^1_{L}(#1,y^a)}
\newcommand{\Qint}{{\cal I}_{H}}

\newcommand{\norm}[1]{\left\lVert#1\right\rVert}

\newcommand{\eps}{\varepsilon}

\newcommand{\avrg}[2]{{\langle #1 \rangle}_{#2}}
\newcommand{\TwoNorm}[2]{{\left\lVert#1\right\rVert}_{L^2\left(#2,y^a\right)}}
\newcommand{\HNorm}[2]{{\left\lVert#1\right\rVert}_{H^1\left(#2,y^a\right)}}

\newcommand{\x}{{\bf x}}

\newcommand{\vnode}{{\bf v}}

\newcommand{\wnode}{{\bf w}}

\newcommand{\tr}[2]{\text{tr}_{#1}\left(   #2    \right)}

\begin{document}

\title{Numerical Homogenization of Heterogeneous Fractional Laplacians}

\author{%
Donald L. Brown \thanks{GeoEnergy Research Center (GERC), University of Nottingham, School of 
Mathematical Sciences,  \mbox{donald.brown@nottingham.ac.uk}}%
\and %
Joscha Gedicke \thanks{Faculty of Mathematics, University of Vienna, 1090 Vienna, Austria, 
\mbox{joscha.gedicke@univie.ac.at}}%
\and %
Daniel Peterseim \thanks{University of Augsburg, Department of Mathematics, 
\mbox{daniel.peterseim@math.uni-augsburg.de}}%
}

\maketitle

\begin{abstract}
In this paper, we develop a numerical multiscale method to solve the fractional Laplacian with a
heterogeneous diffusion coefficient.  When the coefficient is heterogeneous, this adds to the computational 
costs. 
Moreover,  the fractional Laplacian is a nonlocal operator in its standard form, however
the Caffarelli-Silvestre extension allows for a localization of the equations. This adds a complexity of an extra spacial 
dimension and a singular/degenerate coefficient depending on the fractional order. Using a sub-grid correction 
method, we correct the basis functions in a natural weighted Sobolev space and show that these corrections 
are able to be truncated to design a computationally efficient scheme with optimal convergence rates. A key 
ingredient of this method is the use of quasi-interpolation operators to construct the fine scale spaces. 
Since the solution of the extended problem on the critical boundary is of main interest, we construct a  
projective quasi-interpolation  that has both $d$ and $d+1$ dimensional averages over subsets in the spirit of 
the Scott-Zhang operator.  We show that this operator satisfies local stability and local approximation 
properties in weighted Sobolev spaces. We further show that we can obtain a greater rate of 
convergence for sufficient  smooth forces, 
and utilizing a global $L^2$ projection on the critical boundary. 
We present some numerical examples, utilizing our projective quasi-interpolation in 
dimension $2+1$ for analytic and heterogeneous cases to demonstrate the rates and effectiveness of the 
method.
\end{abstract}


\section{Introduction}
In the modeling and simulation of porous media or composite materials, the multiscale nature of the materials 
is a challenging mathematical problem.  In addition to this challenge, 
the modeling of non-local behavior that naturally occurs in particular media
is of great interest, for example in the modeling of 
non-local mechanics \cite{silling2000reformulation}, fractional (and thus  non-local) Keller-Segel  
models of chemotaxis \cite{stinga2014fractional}, and in ground water flow  by fractional (non-Fickian) 
transport \cite{defterli2015fractional,meerschaert2006fractional}.
The areas of multiscale problems and non-local fractional problems have significant overlap in these applications.
In particular, it is well known in hydrology and reservoir engineering that the permeability of the subsurface is 
highly heterogeneous. 
The Macro-Dispersion Experiment (MADE) \cite{rehfeldt1992field} demonstrated experimentally  non-Darcy 
transport that exhibits non-local effects. The challenge of  simulating  these types of problems is two fold: 1) 
the heterogeneity of the subsurface properties creates the need for higher resolutions and 2) the non-locality 
effects the band structure of the linear solvers creating often dense matrices.  In this work, we present a 
multiscale method to mitigate both issues of non-locality and the heterogeneous properties.
\par
The model we will focus on in this work is the heterogeneous fractional Laplacian. This is the Darcy flow model 
with a multiscale permeability coefficient and a fractional derivative power to incorporate the non-local 
behavior. 
There is a vast literature on the analysis and simulation of the fractional Laplacian. Due to a relatively recent 
result of Caffarelli and Silvestre  \cite{caffarelli2007extension}, the solution of the fractional Laplacian 
is more tractable in both terms of analysis and computation. By adding an extra spatial dimension, the 
fractional Laplacian is transformed into an weighted harmonic extension problem, or a singular/degenerate 
(depending on fractional degree) linear elliptic problem. The numerical solution of such problems has been 
approached by several authors in \cite{nochetto2015pde} using quasi-interpolation, in 
\cite{bonito2015numerical} using a novel integral representation formula,  fractional equations are solved via a 
Petrov-Galerkin method in \cite{jin2016petrov}, and for tensor finite elements in \cite{SchwabFractional}, to 
name a few.
A recent survey article of numerical methods for fractional diffusion equations in homogeneous media can be 
found in \cite{bonitoNochetto}.
\par
%
The new challenge to be addressed in this paper is the derivation of effective numerical methods for fractional 
diffusion in heterogeneous media. The application of  numerical homogenization techniques has, to our 
knowledge, not been considered yet. The key idea of numerical homogenization being to incorporate scales on the 
fine-grid to the coarse-grid in a computationally feasible way. Several approaches exist to this end. The 
multiscale finite element method \cite{Hou:Wu:1997}, where local basis functions are computed, the 
heterogeneous multiscale method \cite{Abdulle:E:Engquist:Vanden-Eijnden:2012}, where local problems are 
solved to obtain coarse-grid coefficients, and the variational multiscale method \cite{MR1660141}, which is 
related to the technique we will use. 
We will employ the  local orthogonal decomposition (LOD) method. The LOD method
is a numerical homogenization method whereby  the coarse-grid is augmented so that the corrections are 
localizable and truncated to design a computationally efficient scheme 
\cite{HP13,Kornhuber.Peterseim.Yserentant:2016,MP11,Peterseim:2015}. This has been used successfully 
in many applications such as semi-linear problems \cite{HMP12}, waves equations \cite{Henningwave}, 
perforated media \cite{brown2016multiscale}, and diminishing the pollution in high-frequency problems 
\cite{brown2016multiscaleelastic,Brown2017}, to name a few. 
\par
A key component of this method is a quasi-interpolation operator that is utilized to construct a fine-scale space.
The construction of such an operator for the fractional Laplacian is slightly more delicate due to the extra resolution 
one wants near the trace of the weighted extension problem. 
%
The authors in \cite{nochetto2015pde} utilize a quasi-interpolation based on regularized Taylor polynomials 
\cite{brenner2007mathematical}, which are a generalization of the Cl\'{e}ment quasi-interpolation 
\cite{clement1975approximation}. However, these quasi-interpolations are not projective. We proceed similar 
to \cite{brown2016multiscale}, where the authors utilized a local $L^2$ projection onto the coarse-grid 
space, and prove local $L^2$ stability and approximabilty  properties in weighted Sobolev spaces based on arguments in  
\cite{bramble2002stability,bramble1991some}.  
For the weighted extension problem of the fractional Laplacian we would like 
to further resolve the information on the trace of the original domain. To this end, we develop a 
hybrid projective quasi-interpolation operator using techniques from \cite{bramble2002stability,scott1990finite}, whereby we 
use local $L^2$ projections for both $d$ and $d+1$ 
dimensional simplices to generate nodal values. With this quasi-interpolation, we prove the canonical 
convergence rate of $H^s$, $H$-coarse mesh size and $s$-fractional derivative degree, of the multiscale 
method on quasi-uniform meshes.  Supposing  more smoothness on the data and utilizing a slightly modified projection based on the 
gobal $L^2$ projection on the critical boundary, 
we are able to prove order $H$ convergence on the coarse-grid. We also prove the standard 
estimates with truncated corrections \cite{HMP12}.
\par
We present numerical results for two benchmark examples with the same forcing, but different 
diffusion coefficients, 
in the computational domain that is a subset of $\mathbb{R}^2 \times(0,T)$. The first being a homogeneous  
problem of which has an analytic solution, and the second utilizing a heterogeneous coefficient from a $2-d$ slice of 
a standard benchmark problem. We show that we numerically obtain optimal rates of convergence in these examples 
once we pass the pre-asymptotic regime in terms of the truncation of the correctors. We compute solutions for 
various fractional orders $s$ above and below the critical fractional value of $s=\frac{1}{2}$.
\par
This paper is organized as follows. We begin in Section \ref{prelim} with the heterogeneous fractional 
Laplacian and the singular/degenerate elliptic problem of the Caffarelli-Silvestre extension. 
The weighted extension problem decays exponentially in the extended direction and thus can be truncated on 
a finite domain, this is the problem we shall focus on in this work. In Section \ref{sobolevspaces}, we define the 
relevant fractional Sobolev Spaces for completeness and develop the theory of weighted Sobolev Spaces 
critical to the setup and analysis of the Caffarelli-Silvestre extension. We also present various relevant 
weighted inequalities, such as the weighted Poincar\'{e} inequality. Then, in Section \ref{quasi_int_section}, we 
define the weighted quasi-interpolation operator that will be used to construct the LOD method. Local 
approximability and stability in the weighted spaces are proved. The multiscale method and related errors are 
introduced in Section \ref{MSmethod}. We then present two numerical examples in Section \ref{numerics}. 
Finally, the proofs for the truncation of correctors in weighted norms are given in the Appendix  
\ref{truncproofsection}.

\section{Preliminaries}\label{prelim}
It is  well known that fractional Laplacian problems are  non-local.  Therefore, 
applying standard two-grid techniques to handle heterogeneous coefficients locally
is not possible as the sub-grid problems will too be non-local (in fact global). 
However, due to the Caffarelli-Silvestre extension \cite{caffarelli2007extension}, 
one is able to rewrite the non-local fractional Laplacian 
as a Dirichlet-to-Neumann mapping problem. 
This problem is localizable at the cost of a one dimension higher infinite domain and singular 
or degenerate coefficients depending on the fractional degree $s$. 
In this section we present the 
background on the fractional Laplace operator with a heterogeneous coefficient as well as the background on 
the Caffarelli-Silvestre extension problem.

\subsection{Heterogeneous Fractional Laplacian}
Let $\Omega\subset \mathbb{R}^d$ be a bounded, open, and connected Lipschitz domain for $d\geq 1$.
We let ${\cal L}_{A}u=-\text{div}_x\left(A(x)\nabla_x u \right)$, 
where $A\in \left(L^{\infty}(\dmn)\right)^{d\times d}$ is assumed to be symmetric 
and satisfies for all $x\in \dmn$, $\xi\in\mathbb{R}^d$, and some $\alpha,\beta>0$
\begin{align*}
 \alpha|\xi|^2\leq \avrg{A(x)\xi,\xi}{}\leq \beta |\xi|^2.
\end{align*}
We consider the following fractional Laplace equation with Dirichlet boundary condition, that is we
seek a solution $u$ that satisfies for $s\in (0,1)$ and given data $f$:
\begin{subequations}\label{fracLaplace}
\begin{align}
 \left({\cal L}_{A}u\right)^s&=f \quad\text{in } \dmn,\\
 \label{dirichlet}
 u&=0 \quad\text{on } \partial \dmn.
\end{align}
\end{subequations}
As shown in \cite {caffarelli2016fractional}, one can write the heterogeneous fractional Laplacian  
\eqref{fracLaplace} as 
\begin{align}\label{fracLaplaceoperator}
\left({\cal L}_{A}u\right)^s=\int_{\Omega} (u(x)-u(y))K_{s}(x,y) \,dy,
\end{align}
where $K_s(x,y)$ is the fundamental heat kernel to the operator ${\cal L}_{A}u$, and satisfies the bounds
 $$0\lesssim K_s(x,y)\lesssim \frac{1}{|x-y|^{d+2s}}, \, \, x\neq y.$$
In this work,  we will write  $a\lesssim b$, to mean that there exists a 
constant $C>0$ independent of the mesh 
parameters 
(but possibly depending on the domain, dimension, {$s$, $\alpha$, $\beta$ but not on variations of $A$}) 
such that $a \leq C b$.
Note that for the above integral formulation, one must compute the heat kernel 
for the heterogeneous operator which is computationally  costly. 
\par
The fractional Laplacian operator may also be defined via the eigenfunctions of ${\cal L}_{A}$ given by 
\begin{subequations}\label{spectral}
\begin{align}
-\text{div}_x\left(A(x)\nabla_x \phi_k \right)&=\mu_k \phi_k \quad\text{in } \Omega,\\
\phi_k&=0 \quad\text{on } \partial \Omega,
\end{align}
\end{subequations}
where  the eigenpairs {$(\mu_k ,\phi_k )\in \mathbb{R}_+ \times H^1_0(\Omega)$, for $k\in \mathbb{N}$, can 
be chosen such that $\{\phi_k\}_{ k\in \mathbb{N}}$ form an orthonormal basis for $L^2(\Omega)$}.  Supposing 
$u\in \text{Dom}( {\cal L}^s_{A})$, we expand $u$ as $u(x)=\sum_{k\in \mathbb{N}} u_k \phi_k(x)$, and 
define
\[
{\cal L}^s_{A}u=\sum_{k\in \mathbb{N}} \mu ^s_k u_k \phi_k,
\]
where $u_k=\int_{\Omega} \phi_k u \,dx$. 
 
\subsection{Caffarelli-Silvestre Extension Problem}
Using the formulation developed in \cite{caffarelli2016fractional,nochetto2015pde}, we reformulate the 
fractional Laplacian problem \eqref{fracLaplace} as an extension in  $\dmn\times (0,\infty)\subset \mathbb{R}^{d+1}$.  
We denote the cylinder $\cil=\dmn \times (0,\infty)$, 
spatial variables $x\in \mathbb{R}^d$ and $y\in \mathbb{R}$, 
and the lateral boundary $\partial_{L}\cil=\partial \dmn \times [0,\infty)$.  
We let $U=U(x,y):\cil\to \mathbb{R}$, be a solution to the following singular/degenerate elliptic equation with
coefficients $y^\alpha$:
\begin{subequations}\label{fracLaplace.local}
\begin{align}
  -\text{div}\left( y^a B(x)\nabla U   \right)&=0 \quad\text{in } \cil, \\
  \frac{\partial U}{\partial \nu^{a}}=-y^{a}\frac{\partial U}{\partial y}\Bigg|_{y=0}&=c_s f(x)  \quad\text{on }  \dmn, \\
  U&=0 \quad\text{on }\partial_{L}\cil.
\end{align}
\end{subequations}
The solution to \eqref{fracLaplace} is given by $u(x) = U(x,0)$ for $x\in {\Omega}$.
\par
Above, differential operators are given with respect to $x\in \mathbb{R}^d$ and $y\in \mathbb{R}$, i.e. 
$\nabla=(\nabla_x,\partial_y)^T$,
and the tensor $B\in \mathbb{R}^{d+1}\times  \mathbb{R}^{d+1}$ is given by
\[
B(x)=
\begin{bmatrix}
A(x) & 0_{d\times1} \\
0_{1\times d} & 1
\end{bmatrix},
\]
for $a=1-2s\in (-1,1)$ or $s=\frac{a-1}{2}\in (0,1)$. We will often move freely between the fractional degree $s$ 
and the power of the weight $a$.  Here, $\frac{\partial U}{\partial \nu^{a}}$ is the co-normal exterior derivative 
with outer unit normal $\nu$ and $c_s=2^{1-2s}\frac{\Gamma(1-s)}{\Gamma(s)}>0$ 
is a positive constant that solely depends on $s$.
\par
We note that, supposing appropriate data $f$,  $u$ is a solution of the heterogeneous fractional Laplacian 
\eqref{fracLaplace} if and only if $U$ is a solution to the weighted harmonic extension \eqref{fracLaplace.local}. 
The solution to the weighted harmonic extension is related to the spectral representation of the solution of the 
fractional Laplace. We write $u(x)=\sum_{k\in \mathbb{N}}u_k\phi_k(x)$, where $\{ \phi_k\}_{k\in \mathbb{N}}$ 
satisfy \eqref{spectral}, then we have from \cite{brandle2013concave,capella2011regularity}, that we may write
\[
U(x,y)=\sum_{k\in \mathbb{N}}u_k\phi_k(x) \psi_k(y),
\]
where $\psi_k(y)$ satisfies 
\[
\psi^{''}_k+\frac{a}{y}\psi'_k-\mu_k \psi_k=0 \text{ in } (0, \infty),
\]
with the boundary conditions $\psi_k(0)=1,$ and $\lim_{y\to \infty}\psi_k(y)=0$, for all $k \in \mathbb{N}$.
This above equation has a known solution from \cite{cabre2010positive,capella2011regularity}, that is
$\psi_k(y)=\exp(-\sqrt{\mu_k}y)$ if $s=\frac{1}{2}$ and 
$\psi_k(y)=C_s \left(\sqrt{\mu_k}y \right)^s K_s(\sqrt{\mu_k}y)$, for $s\in (0,1) \backslash \{\frac{1}{2}\}$, where 
$K_s$ is the modified Bessel function of second kind. 
Therefore the solution decreases exponentially in the 
$y$-direction, allowing to truncate the computational domain.
\begin{remark}
Naturally, $f$ is in the dual-space $\mathbb{H}^{-s}(\Omega)$ of the fractional space $\mathbb{H}^{s}$ 
(to be defined more precisely in Section \ref{fractionalspacesection}). However, we will often take 
$f$ to be more regular and suppose $f\in L^2(\Omega)$ or in $\mathbb{H}^{1-s}(\Omega)$ when the extra 
regularity is useful or needed for existence and uniqueness. 
\end{remark}
\begin{remark}
We will further suppose that $f$ is compatible with the Dirichlet boundary condition c.f . \cite[Remark 2.8]
{nochetto2015pde}. In particular, we will suppose that in the regime $s\in (0,\frac{1}{2})$, the data vanishes 
sufficiently fast near $\partial \Omega$, in the regime $s\in (\frac{1}{2},1)$,  $f\in \mathbb{H}^{1-s}(\Omega)$ is 
sufficient. The case $s=\frac{1}{2}$ is the non-weighted standard harmonic extension.
\end{remark}
\par
To facilitate the solution of \eqref{fracLaplace.local} we need additional notation and properties of weighted 
Sobolev spaces, explored in great detail in  \cite{kufner1985weighted}.
For $x\in \mathbb{R}^d$ and $y\in\mathbb{R}_{+}$, we write $\x=(x,y)$ and  let $d\x=dx\,dy$, be the standard 
tensor product Lebesgue measure on $\mathbb{R}^{d+1}$. For $\omega\subset \mathbb{R}^{d}\times 
\mathbb{R}_{+}$, an open set and $a:=1-2s\in(-1,1)$, we define $L^2(\omega,y^a)$ to be all measurable 
functions $u$ on $\omega$ such that
\[
   \TwoNorm{u}{\omega}^2=\int_{\omega} u^2\, y^a \,d\x < \infty,
\]
and define $\Hspace{\omega}$ similarly,  by all measurable functions $u$ on $\omega$ such that
\[
   \HNorm{u}{\omega}:=\left(\TwoNorm{u}{\omega}^2+\TwoNorm{\nabla u}{\omega}^2\right)^{\frac{1}{2}}< \infty.
\]
Finally, we define the space incorporating the homogeneous Dirichlet boundary condition 
on the outer cylinder as 
\[
   \Hdir{\cil}=\{u\in\Hspace{\cil}\, : \, u=0 \text{ on } \partial_L \cil\}.
\]
Integrating \eqref{fracLaplace.local} by parts we obtain the following weak form:  find $U\in \Hdir{\cil}$ such 
that
\begin{align}\label{varform}
B(U,\psi)=F(\psi)\quad  \text{for all }  \psi\in \Hdir{\cil},
\end{align}
where the bilinear and linear forms read
\[
B(U,\psi) :=\int_{\cil} B(x) \nabla U\nabla \psi \,  y^a \,d\x
\qquad\textrm{and}\qquad
F(\psi):=\int_{\dmn   }c_s f(x) \psi(x,0)\,dx.
\]
As the above problem is in an infinite domain, we introduce a truncated cylinder solution 
for computations, which is extend by zero to the infinite domain.
We denote the truncated domain $\cilt=\dmn \times(0,T)$, and 
$\partial_{L}\cilt=\left(\partial \dmn \times[0,T]\right)\cup \left( \Omega\times \{T\}\right)$, for 
some $T>0$. We have the related truncated space given by 
\[
\Hdir{\cilt}=\{u\in\Hspace{\cilt}\, : \, u=0 \text{ on } \partial_L \cilt\}.
\]
We then solve for
$U_T\in \Hdir{\cilt}$ such that
\begin{align}\label{varform.trunc}
B_T(U_T,\psi)=F(\psi)\quad  \text{for all }  \psi\in \Hdir{\cilt},
\end{align}
where we introduce the natural notation for the truncated bilinear form
\[
B_T(U_T,\psi) :=\int_{\cilt}B(x) \nabla U_T\nabla \psi  \, y^a \,d\x.
\]
Extending $U_T$ by zero into $\cil$ we may obtain an infinite domain approximation which we do not relabel.
The following exponential error estimate was proven in \cite[Lemma 3.3]{nochetto2015pde}, which we restate 
here for completeness.
\begin{theorem}
Let $T\geq1$ and let $U$ be a solution to \eqref{varform}, 
and $U_T$ satisfy \eqref{varform.trunc}, for $f\in H^{-s}(\Omega)$,
then we have
\begin{align*}
  \TwoNorm{\nabla( U-U_T)}{\cil}\lesssim e^{-C T}\norm{f}_{H^{-s}(\Omega)},
\end{align*}
for $C>0$ independent of $T$.
\end{theorem}
Thus, the solution of the truncated problem will suffice for a sufficiently large $T$. 
In the remaining parts of this paper, we will merely consider the numerical approximation of $U_{T}$  extended 
by zero into $\cil$.
We will drop the truncation notation in the following sections, as well as the capital lettering $U$ for the solution to the weighted 
harmonic extension if there is no ambiguity.
\begin{remark}
For a full discussion on the regularity and approximation of the fine-grid problem we refer again to  
\cite[Section 2.6]{nochetto2015pde}. For our numerical homogenization method, we will not consider the 
fine-grid error and focus merely on the coarse-grid error.
\end{remark}

\section{Sobolev Spaces and Inequalities}\label{sobolevspaces}
In this section we will introduce the notation of fractional and weighted Sobolev spaces.   First, we recall
the results and notation  of fractional and  weighted Poincar\'{e} inequalities presented in 
\cite{nochetto2015pde} and references therein. We also present and prove some useful inverse and trace 
inequalities in the weighted Sobolev space, thus linking the two kinds of Sobolev spaces.  

\subsection{Fractional Sobolev Spaces}\label{fractionalspacesection}
Here we recall some details of fractional Sobolev spaces as they will be related to the trace spaces of the 
weighted spaces we will consider, as well as being the natural space for the solution $u$ to 
\eqref{fracLaplace}. There is a vast literature on this subject and for details we refer to \cite{di2012hitchhikers}.  
We loosely follow the presentation of  \cite{nochetto2015pde} in the following. 
We begin by introducing the Gagliardo-Slobodeckij seminorm for $s\in(0,1)$:
\begin{align*}
  |u|^2_{H^s(\Omega)}=\int_{\Omega}\int_{\Omega} \frac{|u(x)-u(x')|^2}{|x-x'|^{d+2s}}\,dx\,dx',
\end{align*}
and the related norm
$\norm{u}^2_{H^s(\Omega)}= \norm{u}^2_{L^2(\Omega)}+|u|^2_{H^s(\Omega)}$. 
We define the Sobolev space $H^s(\Omega)$ to be the measurable  
functions such that  $\norm{u}_{H^s(\Omega)}< \infty$. 
For detailed construction we refer to \cite{tartar2007introduction}. 
We define the space $H^s_0(\Omega)$ to be the closure of $C_{0}^\infty(\Omega)$ with respect to the 
norm $\norm{\cdot}_{H^s(\Omega)}$.
\par
If the boundary of $\Omega$ is smooth enough, an interpolation space interpretation is possible 
\cite{lions2014non}. We may write the Sobolev space with $s\in [0,1]$ and $\theta=1-s,$ as the interpolation 
space pair
\begin{align*}
H^s(\Omega)= \left [H^1(\Omega),L^2(\Omega) \right]_{\theta} 
\quad\text{and}\quad	
H^s_0(\Omega)= \left [H^1_0(\Omega),L^2(\Omega) \right]_{\theta}, 
\quad
\theta\neq \frac{1}{2}.
\end{align*}
For the critical case $s=\frac{1}{2}$, this is the so called Lions-Magenes space
\begin{align*}
      H^{\frac{1}{2}}_{00}(\Omega)= \left [H^1_0(\Omega),L^2(\Omega) \right]_{\frac{1}{2}}, 
\end{align*}
this space satisfies 
\begin{align*}
  H^\frac{1}{2}_{00}(\Omega)
  =\Bigg\{u\in H^\frac{1}{2}(\Omega): \int_{\Omega}\frac{u^2(x)}{\text{dist}(x',\partial \Omega )} \,dx' < \infty \Bigg\}.
\end{align*}    
We summarize this in a general notation as
\begin{align*}
\mathbb{H}^s(\Omega)=
  \begin{cases}
  H^s(\Omega), &\text{ for } s\in (0,\frac{1}{2}),\\
  H^{1/2}_{00}(\Omega), &\text{ for } s= \frac{1}{2},\\
  H^s_{0}(\Omega), &\text{ for } s\in (\frac{1}{2},1).
  \end{cases}
\end{align*}

\subsection{Weighted Sobolev Spaces and Inequalities}
We now give the background for weighted Sobolev spaces as well as present some critical inequalities. 
A key property of the weight $y^a$ is that it belongs the Muckenhoupt class $A_2(\mathbb{R}^{d+1})$ 
\cite{gol2009weighted,muckenhoupt1972weighted}. For a general weight,
$w\in L_{loc}^1(\mathbb{R}^{d+1})$, we say that $w\in A_{2}(\mathbb{R}^{d+1})$ 
if there exists a $C_{2,w}>0$ such that  
\begin{align}\label{Muckenhouptconstant}
   \sup_{B} \left(\frac{1}{|B|}\int_{B} w \, d\x  \right)\left(\frac{1}{|B|}\int_{B} w^{-1} \, d\x \right) =C_{2,w}<\infty,
\end{align}
for all balls $B\subset\mathbb{R}^{d+1}.$ We will denote the Muckenhoupt weight constant for $y^a$ as 
$C_{2,a}$. We will now give a few of the critical inequalities and properties related to this class of  weighted 
Sobolev spaces.
\par
A key inequality for the analysis is the weighted Poincar\'{e} inequality.
The weighted Poincar\'{e} inequality for Muckenhoupt weights is well studied in nonlinear potential theory of 
degenerate problems \cite{fabes1982local,heinonen2012nonlinear} and references therein. We will state  the 
result here without proof.
\begin{lemma}[Weighted Poincar\'{e} Inequality]\label{poincare2}
Let $\omega\subset\Omega \times (0,\infty)$, be a bounded, 
star-shaped domain (with respect to the ball B) and $\text{diam}(\omega)\approx H$.
If $w\in H^1(\omega, y^a)$ it holds that 
\begin{align}\label{poincare_unit2}
	\norm{w-\avrg{w}{\omega}}_{L^2(\omega,y^a)  }\lesssim H \norm{\nabla w}_{L^2(\omega,y^a ) },
\end{align}
where the constants are independent of $H$ and 
$\avrg{w}{\omega}=\frac{1}{|\omega|}\int_{\omega} w\,d\x$. \qed
\end{lemma}
\begin{remark}
Note that the above inequality may be extended to a connected union of star-shaped domains where the 
average can be taken over a subdomain \cite[Corollary 4.4]{nochetto2015pde}.  We will refer to both of these 
results simply as the Weighted Poincar\'{e} Inequality when there is no ambiguity. 
\end{remark}
\par
We have the following $L^\infty \to L^2$ weighted inverse inequality. For this we suppose that we have a 
coarse quasi-uniform, shape-regular, discretization ${\cal T}_{\cilt}$ of the  domain $\cilt$ with characteristic 
mesh size $H$. Similarly, we denote the restricted mesh onto the lower dimensional space $\Omega$, to be  
${\cal T}_{\Omega}$. We denote by $\mathbb{P}_1(T)$ the linear polynomials on $T\in {\cal T}_{\cilt}$.
\begin{proposition}\label{infitytoL2}
For $p\in \mathbb{P}_1(T)$, we have
\begin{align}\label{inftytoL2}
	\norm{p}_{L^\infty(T ) }&\lesssim  |T|^{-1}\norm{y^{-\frac{a}{2}}}_{L^2(T) }  \norm{p}_{L^2(T,y^a ) }.
\end{align}
\end{proposition}
\begin{proof}
We begin by 
utilizing the following result from classical FE  inverse inequalities
\[
  \norm{p}_{L^r(T) }\lesssim  |T|^{\left(\frac{1}{r}-\frac{1}{q}\right)}\norm{p}_{L^q (T ) }
  \quad\text{for } 1\leq q\leq r<\infty.
\]
For $r=\infty,q=1$, we obtain 
\begin{align*}
\norm{p}_{L^\infty(T ) }&\lesssim  |T|^{-1}\norm{p}_{L^1 (T ) }=|T|^{-1}\int_{T}|py^{a/2}|y^{-a/2}\,d\x
\leq |T|^{-1} \norm{p}_{L^2(T,y^a ) }  \norm{y^{-\frac{a}{2}}}_{L^2(T) } . \qedhere
\end{align*}
\end{proof}
\par
Let $\tr{}{\cdot}$ denote the canonical trace operator for the space $\Hdir{\cil}$, and trivially also the 
zero-extension truncated space $\Hdir{\cilt}$. We state the following trace lemma.
\begin{lemma}\label{CanonicalTrace}
For $u\in \Hdir{\cilt}$, we have 
$\tr{}{u}\in \mathbb{H}^{s}(\Omega)$, and
\begin{align}\label{CanonicalTraceestimate}
    \norm{u}_{\mathbb{H}^s(\Omega)}\lesssim \TwoNorm{  u}{ \cilt }+\TwoNorm{  \nabla u}{ \cilt }.
\end{align}
Thus, $\Hdir{\cilt}\subset \mathbb{H}^s(\Omega)$.
\end{lemma}
\begin{proof}
See \cite{cabre2010positive}, for $s = \frac{1}{2}$, and \cite{lions1959theoremes}, 
for $s\in (0,1) \backslash\{\frac{1}{2}\}$. 
For a general discussion on trace spaces of weighted spaces we refer the reader to 
\cite{nekvinda1993characterization}.
\end{proof}
\begin{remark}
Note that $L^1(\Omega)$ is the canonical  trace space for $W^{1,1}(\cilt)$ \cite{adams2003sobolev}, 
and  by a trivial argument
\begin{align}\label{embedding}
  \norm{ u}_{L^1(\cilt ) }  
  = \norm{ u y^{\frac{a}{2}}y^{-\frac{a}{2}}  }_{L^1(\cilt ) }
  \lesssim C_{-a}(\cilt) \TwoNorm{u}{\cilt},
\end{align}
where $C^2_{-a}(\cilt)=\int_{\cilt} y^{-a}\,d\x$ which is finite on a bounded domain.  Similarly, the result holds for 
$\nabla u$. Thus, we have the embeddings $\Hdir{\cilt}\subset W^{1,1}(\cilt) \subset L^1(\Omega).$ This $L^1$ 
embedding structure suggests the use of quasi-interpolation operators of the Scott-Zhang \cite{scott1990finite} type 
which is discussed in Section \ref{quasi_int_section}.
\end{remark}
\par
We have the following  trace inequalities for elements 
$T\in {\cal T}_{\cilt}$, and faces (edges) $F\in {\cal T}_{\Omega}$. 
\begin{lemma}\label{traceHlemma.L1}
Let $T\in {\cal T}_{\cilt}$ and $F=\partial T\cap \Omega$ be the face (edge) adjacent to $\Omega$. 
Then, for $u\in W^{1,1}(T)$, we have the following inequality 
\begin{align}\label{traceboundH.L1}
	\norm{ u}_{L^1 (F) }   \lesssim |F||T|^{-1}\left( \norm{  u}_{L^1(T)} +   H \norm{  \nabla u}_{L^1(T)}\right).
\end{align}
\end{lemma}
\begin{proof}
This is an application of the trace inequality and  
scaling arguments c.f. \cite[Section 2.4]{melenkApel}. 
\end{proof}
We also have the following weighted trace inequality.
\begin{lemma}\label{traceHlemma}
Let $T\in {\cal T}_{\cilt}$, $F=\partial T\cap \Omega$ be the face (edge) adjacent to $\Omega$,
and $u\in\Hdir{T}$.
Then we have the following inequality 
\begin{align}\label{traceboundH}
  \norm{ u}_{L^2 (F) }   
      \lesssim H^{s-1} \TwoNorm{u}{T} +H^{s}  \TwoNorm{\nabla u}{T}.
\end{align}
\end{lemma}
\begin{proof}
We proceed by using mapping arguments similar to \cite[Lemma 7.2]{ErnGuermond2015} 
and weighted-scaling arguments from \cite{d2012finite,d2008coupling}. 
We prove the result for a  simplex  $T\in {\cal T}_{\cilt}$, 
such that $F:=\partial T\cap \Omega$ is a face or edge (not a vertex only).
We denote the reference (unit size) element $\hat{T}$ 
and similarly the reference boundary face $\hat{F}$. 
We let $A_{T}:\hat{T}\to T$ be an affine mapping, 
and denote $\hat{u}=u\circ A_{T}$, $\hat{\x}=A_{T}^{-1}(\x)$, for $\x\in T$, 
and $\text{diam}(T)\approx\text{diam}(F)\approx H$.
Note that from \cite[Lemma 3.2]{d2012finite} and from shape regularity we have that 
$(A_{T}(\hat{y}))^a\geq C H^a \hat{y}^a$, 
thus, 
\begin{align}\label{scalingK}
  \TwoNorm{u}{T}^2
    =\int_{T} u^2 y^a \,d\x
    =\frac{|T|}{|\hat{T}|}\int_{\hat{T}} \hat{u}^2(\hat{\x}) (A_{T}(\hat{y}))^a d\hat{\x} 
    \geq C H^a \frac{|T|}{|\hat{T}|} \norm{\hat{u}}^2_{L^2(\hat{T},\hat{y}^a )}.
\end{align}
By using standard trace inequality arguments,
the trace bound \eqref{CanonicalTraceestimate}, 
and the above scaling \eqref{scalingK}, 
in the weighted norm we obtain 
\begin{align*}
\norm{{u}}_{L^2(F)}
&=\left( \frac{|F|}{|\hat{F}|} \right)^{\frac{1}{2} }  \norm{{\hat{u}}}_{L^2(\hat{F})}
  \lesssim |F| ^{\frac{1}{2} }  \left(  {\left\lVert\hat{u}\right\rVert}_{L^2\left(\hat{T},\hat{y}^a\right)}
  +{\left\lVert\nabla\hat{u}\right\rVert}_{L^2\left(\hat{T},\hat{y}^a\right)} \right)\\
&\lesssim |F| ^{\frac{1}{2} } |T|^{-\frac{1}{2}}H^{-\frac{a}{2}} \left(  \TwoNorm{  {u}}{ {T} }+\norm{\nabla A_{T}}\TwoNorm{  \nabla {u}}{ {T} }\right)\\
&\lesssim H ^{-\frac{ 1}{2} }  H^{-\frac{a}{2}} \left(  \TwoNorm{  {u}}{ {T} }+H\TwoNorm{  \nabla {u}}{ {T} }\right).
\end{align*}
Thus, with $a=1-2s$ we obtain the estimate \eqref{traceboundH}. 
\end{proof}
We also have the Poincar\'{e} inequality in the non-weighted trace space $\mathbb{H}^s(\Omega)$.
\begin{lemma}\label{tracePoincare}
Let $T\in {\cal T}_{\cilt}$, $F=\partial T\cap \Omega$ be the face (edge), 
and $u\in \mathbb{H}^s(F)$.
Then, we have the following inequality 
\begin{align*}
	\norm{ u-\avrg{u}{F}}_{L^2 (F) }   \lesssim H^s\norm{u}_{\mathbb{H}^s(F)}.
\end{align*}
\end{lemma}
\begin{proof}
This can be seen in \cite[Lemma 7.1]{ErnGuermond2015}. 
\end{proof}
\par
Finally, we will need the Caccioppoli inequality
for truncation arguments of the sub-grid correctors in Appendix \ref{truncproofsection}.
Here we recall the Caccioppoli inequality presentation as in \cite{caffarelli2016fractional}. 
Let $B_{r}(x_0)$ be the $r$-ball in $\mathbb{R}^d$, centered at $x_0$ 
and define the cylinder $B_{r}(x_0)^*=B_{r}(x_0)\times (0,r)\subset \cilt$. 
Choosing $x_0=0$ and suppressing this notation we consider the following problem: 
Find $u\in \Hspace{B_1^*}$ such that
\begin{subequations}\label{fracLaplace.cacc}
\begin{align}
  \text{div}\left( y^a B(x)\nabla u   \right)&=\text{div}\left( y^a g \right) \quad\text{in } B_1^*, \\
  -y^{a}\frac{\partial u}{\partial y}\Bigg|_{y=0}&=f  \quad\text{on }  B_1,
\end{align}
\end{subequations}
with $g_{i}\in L^2(B_1^*,y^a)$, $i=1,\dots, d$, and $g_{d+1}=0$. 
Suppose without loss of generality that 
$B(0)={ I}$, then we have the following lemma.
\begin{lemma}[Caccioppoli Inequality]\label{Caccioppoli}
Let $u$ be a weak solution to \eqref{fracLaplace.cacc}, 
then for $\eta\in C^{\infty}(\overline{B_1^*})$, that vanishes on
$\partial B_{1}^* \backslash \overline{B_1} $ we have
\begin{align}\label{cacc.ineq}
    \int_{B_{1}^*}y^a \eta^2 |\nabla u|^2\,d\x
    \lesssim \int_{B_{1}^*}y^a\left( |\nabla \eta |^2 u^2 +|g|^2\eta^2\right)\,d\x+\int_{B_1} (\eta(x,0))^2 |u(x,0)| |f(x)| \,dx.
\end{align}
\end{lemma}
\begin{proof}
See  \cite[Lemma 3.2]{caffarelli2016fractional}
\end{proof}
\begin{remark}Note that away from the critical boundary, the standard Caccioppoli inequality will also hold due 
to the boundedness of the weight $y^a$ on bounded domains.
\end{remark}

\section{Quasi-Interpolation in Weighted Sobolev Spaces}\label{quasi_int_section}
Here we construct a quasi-interpolation operator for weighted Sobolev spaces using a hybrid of local $L^2$ 
projections onto $d$ and $d+1$ dimensional simplices \cite{bramble2002stability,scott1990finite}. 
We begin by introducing the discretization with a classical nodal basis. From here we are able to build a 
quasi-interpolation based on local weighted $L^2$ projections. The novelty here being that we do not only include 
the weighted spaces, but also augment the quasi-interpolation on the critical trace $\Omega$. We have two 
types of local $L^2$ projections, one onto the nodes of the cylinder domain $\cilt$ and a lower dimensional 
projection onto nodes on $\Omega$.
We then state the local stability and approximability properties of these operators both in the interior of the 
domain and for the canonical traces. We utilize arguments of proof along the lines of \cite{melenkApel}.  

\subsection{Classical Nodal Basis}
The key idea is that the resulting quasi-interpolation is stable in the weighted Sobolev norm, 
and stable on $\Omega$ in the lower regularity space $\mathbb{H}^s(\Omega)$. 
Following much of the notation in \cite{MP11}, recall that we suppose that we have a coarse quasi-uniform, 
shape-regular discretization ${\cal T}_{\cilt}$ of the domain $\cilt$ with characteristic mesh size $H$. 
Similarly, we denote the restricted mesh onto the lower dimensional space
$\Omega$, to be  ${\cal T}_{\Omega}$.
\par
We denote all the nodes of the mesh as ${\cal N}$. The interior   nodes of ${\cal T}_{\cilt}$ 
(not including nodes on $\Omega$, nor vanishing Dirichlet condition) 
we denote as ${\cal N}_{int}$, the free nodes on $\Omega$ are denoted as 
${\cal N}_{\Omega}$, and the Dirichlet nodes as ${\cal N}_{dir}$. Also, it will be 
useful to combine all the nodes with degrees of freedom, 
we denote those as ${\cal N}_{dof}={\cal N}_{int}\cup{\cal N}_{\Omega}$.
We will write ${\cal N}(\omega)$ for nodes  in  $\overline{\omega}$, 
similarly for interior, boundary, or Dirichlet nodes.
Let the classical conforming $\mathbb{P}_{1}$
finite element space over ${\cal T}_{\cilt}$ be given by $S_{H}$,
and let $V_{H}=S_{H}\cap \Hdir{\cilt}$. 
Utilizing the notation in \cite{nochetto2015pde},  we denote $\vnode\in{\cal N}_{}$ as nodal values. 
The $\mathbb{P}_1$ nodal basis functions $\lambda_{\vnode}$, for all nodes $\vnode \in{\cal N}$,
form a basis for $V_{H}$, and are defined for a node $\vnode \in{\cal N}$ as
\begin{align}\label{P1}
    \lambda_{\vnode}(\vnode)=1 
    \text{ and } 
    \lambda_{\wnode}(\vnode)=0, \vnode\neq \wnode\in {\cal N}.
\end{align}
We define the  patch around $\vnode$ as
\[
\omega_{\vnode}=\bigcup_{T \ni \vnode}T,
\]
for $T\in {\cal T}_{\cilt}$.
Using the definition and notation in \cite{Henning.Morgenstern.Peterseim:2014}, 
we define for any patch $\omega_{\vnode}$
the extension patch
\begin{subequations}\label{patches}
\begin{align}
	{\omega}_{\vnode}&={\omega}_{\vnode,0}=\text{supp}( \lambda_{\vnode}) \cap \cilt,\\
	{\omega}_{\vnode,k}&=\text{int}(\cup \{ T\in {\cal T}_{\cilt}| T\cap \bar{\omega}_{\vnode,k-1}\neq \emptyset \}\cap \cilt,
\end{align}
\end{subequations}
for $k\in \mathbb{N}_{+}$.
Suppose that for these patches $\frac{|B|}{|{\omega}_{\vnode,k}|}\lesssim 1$ for some ball $B$ containing 
${\omega}_{\vnode,k}$,
thus we have the bound 
\begin{align}
    &\left(\frac{1}{|{\omega}_{\vnode,k}|}\int_{{\omega}_{\vnode,k}} w \, d\x  \right)
    \left(\frac{1}{|{\omega}_{\vnode,k}|}\int_{{\omega}_{\vnode,k}} w^{-1} \, d\x \right) 
     \lesssim \left(\frac{|B|}{|{\omega}_{\vnode,k}|}\frac{1}{|B|} \int_{B} w \, d\x  \right)
     \left(\frac{|B|}{|{\omega}_{\vnode,k}|}\frac{1}{|B|} \int_{B} w^{-1} \, d\x \right)\nonumber \\
   &\qquad \lesssim \left(\frac{|B|}{|{\omega}_{\vnode,k}|  } \right)^2 \left(\frac{1}{|B|} \int_{B} w \, d\x  \right)
   \left(\frac{1}{|B|} \int_{B} w^{-1} \, d\x \right) 
   \lesssim \left(\frac{|B|}{|{\omega}_{\vnode,k}|  } \right)^2  C_{2,w} 
   \lesssim C_{2,w},  \label{Muckenhouptconstant.patch}
\end{align}
where we utilized the bound \eqref{Muckenhouptconstant}.
Hence, we can apply the Muckenhoupt weight 
bounds to the patches $\omega_{\vnode,k}$.
\par
We will also need to define the boundary-$\Omega$ patches. Let $\vnode \in {\cal N}_\Omega$, and we take 
$\tr{}{\lambda_{\vnode}}=\lambda_\vnode(x,0),$ and denote 
\begin{subequations}\label{boundarypatches}
\begin{align}
  \partial{\omega}_{\vnode}
  &=\partial {\omega}_{\vnode,0}=\text{supp}( \tr{}{\lambda_{\vnode}}  ) \cap \Omega,\\
  \partial {\omega}_{\vnode,k}
  &=\text{int}(\cup \{ T\in {\cal T}_{\Omega}| T\cap \overline{\partial \omega}_{\vnode,k-1}\neq \emptyset \}\cap \Omega.
\end{align}
\end{subequations}
We will denote $V_{H}|_{\omega}$ to be the coarse-grid space restricted to some domain $\omega$. 

\subsection{Quasi-Interpolation Operator}
The authors in \cite{nochetto2015pde} construct a quasi-interpolation based on a higher order Cl\'{e}ment type 
of operator.
However, in this section, we develop a quasi-interpolation operator that is also a projection in the weighted Sobolev 
space and satisfies beneficial properties on the trace.  This projective quasi-interpolation gives stability 
properties required for the localization theory.
This is a  modification of the operator of \cite{bramble2002stability} and was utilized in perforated domains in 
\cite{brown2016multiscale}. 
Here we adapt this technique to the $y^a$-weighted setting with a slight modification of the $\Omega$ 
boundary terms in the flavor of Scott-Zhang \cite{scott1990finite}.
\par
We now define the two  local  weighted $L^2$ projections. For $\vnode\in{\cal N}_{int}$,
${\cal P}_{\vnode}: L^2({\omega}_{\vnode},y^a)\to {V}_{H}|_{\omega_{\vnode}}$ is 
the local projection operator such that
\begin{align}\label{L2proj}
\int_{{\omega}_{\vnode}}( {\cal P}_{\vnode}u) v_{H} y^a \,d\x= \int_{{\omega}_{\vnode}}u v_{H} y^a \,d\x 
\quad\text{for all } v_{H}\in  {V}_{H}|_{\omega_{\vnode}},
\end{align}
and for  $\vnode\in{\cal N}_{\Omega}$, 
${\cal P}^{\Omega}_{\vnode}: L^2({\partial \omega}_{\vnode})\to {V}_{H}|_{\partial \omega_{\vnode}}$
is the boundary operator such that
\begin{align}\label{L2proj.bd}
\int_{{\partial \omega}_{\vnode}}( {\cal P}^{\Omega}_{\vnode}u) v_{H} dx
= \int_{{\partial \omega}_{\vnode}}u v_{H}  dx 
\quad\text{for all } v_{H}\in  {V}_{H}|_{\partial \omega_{\vnode}}.
\end{align}
From this we define the quasi-interpolation operator 
$\Qint: \Hdir{\cilt} \to {V}_{H}$ for $u\in \Hdir{\cilt}$ as
\begin{align}\label{proj}
\Qint u(\x)=\sum_{\vnode\in{\cal N}_{int}} ({\cal P}_{\vnode} u)(\vnode) \lambda_{\vnode}(\x)+
\sum_{\vnode\in{\cal N}_{\Omega}} ({\cal P}^{\Omega}_{\vnode} u)(\vnode) \lambda_{\vnode}(\x).
\end{align}
\begin{remark}
Note that for a node $\vnode \in {\cal N}_{dir}$, i.e. on $\partial_L \cilt$, the local $L^2$ boundary 
projection operator maybe defined  as 
\begin{align}\label{L2proj.bd.dir}
	\int_{{ \omega}_{\vnode}\cap \partial_L \cilt  }( {\cal P}^{\partial_L \cilt}_{\vnode}u) v_{H} y^{a}  \,dy
	= \int_{{ \omega}_{\vnode}\cap \partial_L \cilt  }u v_{H} y^a   \,dy 
	\quad\text{for all } v_{H}\in  {V}_{H}|_{{ \omega}_{\vnode}\cap \partial_L \cilt  }.
\end{align}
However, $({\cal P}^{\partial_L \cilt }_{\vnode}u)(\vnode)=0,$ since $u=0$. Thus, we take the sum over all the 
nodes, unlike the case of utilizing a $d+1$ dimensional  operator also on the boundary, where
$({\cal P}_{\vnode}u)(\vnode)\neq 0$ for $\vnode \in {\cal N}_{dir}$. 
This simplifies the analysis of the quasi-interpolation operator near the Dirichlet boundary slightly. 
\end{remark}

\subsection{Local Stability and Approximability}
We have the following stability and local approximation properties
of the quasi-interpolation operator $\Qint$ defined by 
\eqref{proj}.
The proof of this lemma is based on that presented in  \cite{melenkApel}.
\begin{lemma}\label{stablelemma}
Let $\Qint$ be given by \eqref{proj} and $\vnode\in {\cal N}$. 
The quasi-interpolation satisfies the following stability estimate for all $u\in  \Hdir{\cilt}$
\begin{subequations}
\begin{align}
  \TwoNorm{\Qint u}{{\omega}_{\vnode}  }
  &\lesssim \TwoNorm{ u}{{\omega}_{\vnode,1}  }+H\TwoNorm{\nabla u}{{\omega}_{\vnode,1}  }, \\ 
  \TwoNorm{\nabla \Qint u}{{\omega}_{\vnode}  }
  &\lesssim \TwoNorm{\nabla u}{{\omega}_{\vnode,1}}. \label{stableproj0}
\end{align}
\end{subequations}
Further, the following approximation estimates hold
\begin{subequations}\label{stableproj}
\begin{align}
  \TwoNorm{u-\Qint u}{{\omega}_{\vnode}  } 
  &\lesssim    H \TwoNorm{\nabla  u}{{\omega}_{\vnode,1} }, \label{stableproj1}\\
  \TwoNorm{\nabla (u-\Qint u)}{ {\omega}_{\vnode} } 
  &\lesssim  \TwoNorm{\nabla  u}{\omega_{\vnode,1}} \label{stableproj2}.
\end{align}
\end{subequations}
Moreover, the quasi-interpolation $\Qint$ is a projection.
\end{lemma}
\begin{proof}
With the  quasi-interpolant \eqref{proj} including the Dirichlet nodes  it has the same property as 
Scott-Zhang \cite{scott1990finite} of preserving the vanishing  Dirichlet boundary conditions. Thus, we 
implicitly sum over the Dirichlet nodes in what follows and need not take special care of boundary nodes as in 
Cl\'{e}ment quasi-interpolation.
\par
In the first case, suppose that
$\vnode'\in {\cal N}_{int}(\omega_\vnode)$ is an interior node, 
then noting that ${\cal P}_{\vnode'}u$ is finite dimensional and
using Proposition \ref{infitytoL2}, we arrive at
\begin{align*}
    \norm{ {\cal P}_{\vnode'} u}_{L^{\infty}(\omega_{\vnode'})}
    &\lesssim  |\omega_{\vnode'}|^{-1}\norm{ {\cal P}_{\vnode'} u}_{L^1 (\omega_{\vnode'}) }
    \lesssim   |\omega_{\vnode'}|^{-1}\left(\int_{\omega_{\vnode'}}y^{-a}\,d\x\right)^{\frac{1}{2}}\norm{ {\cal P}_{\vnode'} u}_{L^2(\omega_{\vnode'},y^a ) }.
\end{align*}
From \eqref{L2proj}, letting $v_H={\cal P}_{\vnode'} u$, we get
\begin{align*}
  &\norm{ {\cal P}_{\vnode'} u}^2_{L^2(\omega_{\vnode'},y^a ) }
  =\int_{{\omega}_{\vnode'}} |{\cal P}_{\vnode'}u|^2 y^a \,d\x
  = \int_{{\omega}_{\vnode'}}u ({\cal P}_{\vnode'}u) y^a \,d\x
  \leq \norm{u }_{L^1(\omega_\vnode', y^a)}\norm{ {\cal P}_{\vnode'} u}_{L^{\infty}(\omega_{\vnode'})}.
\end{align*}
Thus,  manipulating the two above identities yields
\begin{align*}
  \norm{ {\cal P}_{\vnode'} u}_{L^{\infty}(\omega_{\vnode'})}^2 
  &\lesssim |\omega_{\vnode'}|^{-2}\left(\int_{\omega_{\vnode'}}y^{-a}\,d\x\right)\
      \norm{ {\cal P}_{\vnode'} u}^2_{L^2(\omega_{\vnode'},y^a ) }\\
  & \lesssim  |\omega_{\vnode'}|^{-2}\left(\int_{\omega_{\vnode'}}y^{-a}\,d\x\right)
     \norm{u }_{L^1(\omega_\vnode', y^a)}	\norm{ {\cal P}_{\vnode'} u}_{L^{\infty}(\omega_{\vnode'})},
\end{align*}
and so, by taking a larger patch we have 
\begin{align}\label{nodeinfty}
  \left | {\cal P}_{\vnode'} u(\vnode')\right| \lesssim   |\omega_{\vnode,1}|^{-2}
  \left(\int_{\omega_{\vnode,1}}y^{-a}\,d\x\right)^{}\norm{u }_{L^1(\omega_{\vnode,1}, y^a)}.
\end{align}
\par
In the second case, suppose that $\vnode'\in {\cal N}_{\Omega}(\omega_\vnode)$
is a node on the boundary $\Omega$, and so we use the local (unweighted)
$L^2$ projection on the boundary given by \eqref{L2proj.bd}. 
Again, noting that ${\cal P}^{\Omega}_{\vnode'}u$ is finite dimensional and  using an inverse inequality
we get
\begin{align*}
  \norm{ {\cal P}^{\Omega}_{\vnode'} u}_{L^{\infty}(\partial \omega_{\vnode'})}
  &\lesssim |\partial \omega_{\vnode'}|^{-\frac{1}{2}} \norm{ {\cal P}^{\Omega}_{\vnode'} u}_{L^2(\partial \omega_{\vnode'}) }.
\end{align*}
From \eqref{L2proj.bd}, we obtain 
\begin{align*}
    \norm{ {\cal P}^{\Omega}_{\vnode'} u}^2_{L^2(\partial \omega_{\vnode'} ) }
    &=\int_{\partial {\omega}_{\vnode'}} |{\cal P}^{\Omega}_{\vnode'}u|^2 \,dx
    = \int_{\partial {\omega}_{\vnode'}}u ({\cal P}^{\Omega}_{\vnode'}u) \,dx
    \leq \norm{u }_{L^1(\partial \omega_\vnode' )}\norm{ {\cal P}^{\Omega}_{\vnode'} u}_{L^{\infty}(\partial \omega_{\vnode'})}.
\end{align*}
Thus,  again manipulating the two above identities yields
\begin{align*}
  \norm{ {\cal P}^{\Omega}_{\vnode'} u}_{L^{\infty}(\partial \omega_{\vnode'})}^2 
  &\lesssim |\partial \omega_{\vnode'}|^{-1} \norm{ {\cal P}^{\Omega}_{\vnode'} u}^2_{L^2(\partial \omega_{\vnode'} ) }
  \lesssim  |\partial \omega_{\vnode'}|^{-1} \norm{u }_{L^1(\partial \omega_\vnode')}	
   \norm{ {\cal P}^{\Omega}_{\vnode'} u}_{L^{\infty}(\partial \omega_{\vnode'})},
\end{align*}
and so, by taking a larger patch and utilizing the trace inequality \eqref{traceboundH.L1} we obtain 
\begin{align}\label{nodeinfty.bd}
  \left | {\cal P}^{\Omega}_{\vnode'} u(\vnode')\right| 
  &\lesssim |\partial \omega_{\vnode,1}|^{-1} \norm{u }_{L^1(\partial \omega_{\vnode,1})} 
  \lesssim | \omega_{\vnode,1}|^{-1} \left( \norm{  u}_{L^1(\omega_{\vnode,1})} + H\norm{\nabla u}_{L^1(\omega_{\vnode,1})}\right).
\end{align}
\par
Finally, we note that by taking a larger patch $\omega_{\vnode,1}$, we have
\begin{align}\label{basisfunctionweighted}
  \TwoNorm{\lambda_{\vnode'}}{\omega_{\vnode}}
  \lesssim \left(\int_{\omega_{\vnode,1}}y^{a}\,d\x\right)^{\frac{1}{2}}, 
  \quad\text{and}\quad
  \TwoNorm{\nabla \lambda_{\vnode'}}{\omega_{\vnode}}\lesssim H^{-1} \left(\int_{\omega_{\vnode,1}}y^{a}\,d\x\right)^{\frac{1}{2}}.
\end{align}
For the quasi-interpolation $\Qint(u)$ we have
\[
\Qint(u)=\sum_{\vnode'\in{\cal N}_{int}(\omega_\vnode)} 
({\cal P}_{\vnode'} u)(\vnode') \lambda_{\vnode'}
+\sum_{\vnode'\in{\cal N}_{\Omega}(\omega_\vnode)} 
({\cal P}^{\Omega}_{\vnode'} u)(\vnode') \lambda_{\vnode'} \quad\text{in }\omega_\vnode.
\]
For the $L^2$ stability we note that from \eqref{nodeinfty}--\eqref{basisfunctionweighted} we get
\begin{align}\label{stability.QI1}
\TwoNorm{\Qint(u)}{\omega_\vnode}
  &\leq \sum_{\vnode'\in{\cal N}_{int}(\omega_\vnode)} 
    \left|({\cal P}_{\vnode'} u)(\vnode')\right| 
    \TwoNorm{ \lambda_{\vnode'}}{\omega_\vnode}+\sum_{\vnode'\in{\cal N}_{\Omega}(\omega_\vnode)} 
    \left|({\cal P}^{\Omega}_{\vnode'} u)(\vnode')\right| 
    \TwoNorm{\lambda_{\vnode'}}{\omega_{\vnode}} \nonumber \\
  &\lesssim |\omega_{\vnode,1}|^{-2}
    \left(\int_{\omega_{\vnode,1}}y^{-a}\,d\x\right)
    \left(\int_{\omega_{\vnode,1}}y^{a}\,d\x\right)^{\frac{1}{2}}
    \norm{u }_{L^1(\omega_{\vnode,1}, y^a)}\nonumber \\
  &+| \omega_{\vnode,1}|^{-1}
  \left(\int_{\omega_{\vnode,1}}y^{a}\,d\x\right)^{\frac{1}{2}} 
  \left(  \norm{  u}_{L^1(\omega_{\vnode,1})} + H \norm{  \nabla u}_{L^1(\omega_{\vnode,1})}\right).
\end{align}
Now we analyze each part carefully. Note that 
\begin{align}\label{L1-L2estimate}
\norm{u }_{L^1(\omega_{\vnode,1}, y^a)}
\lesssim \left(\int_{\omega_{\vnode,1}}y^{a}\,d\x\right)^{\frac{1}{2}} \TwoNorm{ u}{\omega_{\vnode,1}}.
\end{align}
Since $y^\alpha$ belongs to the Muckenhoupt class $A_2(\mathbb{R}^{d+1})$,
we get from \eqref{Muckenhouptconstant.patch}
\begin{align*}
&
  |\omega_{\vnode,1}|^{-2}\left(\int_{\omega_{\vnode,1}}y^{-a}\,d\x\right)
  \left(\int_{\omega_{\vnode,1}}y^{a}\,d\x\right)^{\frac{1}{2}}\norm{u }_{L^1(\omega_{\vnode,1}, y^a)} \\
&\qquad\lesssim 
  |\omega_{\vnode,1}|^{-2}\left(\int_{\omega_{\vnode,1}}y^{-a}\,d\x\right)
  \left(\int_{\omega_{\vnode,1}}y^{a}\,d\x\right)^{}\TwoNorm{ u}{\omega_{\vnode,1}}
   \lesssim C_{2,a} \TwoNorm{ u}{\omega_{\vnode,1}}.
\end{align*}
For the second term we use 
\eqref{L1-L2estimate} again, also for the derivative terms, thus  
\begin{align*}
&|\omega_{\vnode,1}|^{-1}
  \left(\int_{\omega_{\vnode,1}}y^{a}\,d\x\right)^{\frac{1}{2}} 
  \left(  \norm{  u}_{L^1(\omega_{\vnode,1})} + H \norm{  \nabla u}_{L^1(\omega_{\vnode,1})}\right)\\
&\qquad\lesssim | \omega_{\vnode,1}|^{-1}
  \left(\int_{\omega_{\vnode,1}}y^{a}\,d\x\right)^{\frac{1}{2}} \left(\int_{\omega_{\vnode,1}}y^{-a}\,d\x\right)^{\frac{1}{2}} 
  \left( \TwoNorm{ u}{\omega_{\vnode,1}} + H  \TwoNorm{ \nabla u}{\omega_{\vnode,1}}\right)\\
&\qquad\lesssim  C_{2,a} \left( \TwoNorm{ u}{\omega_{\vnode,1}} + H  \TwoNorm{ \nabla u}{\omega_{\vnode,1}}\right).
\end{align*}
Returning to \eqref{stability.QI1} we obtain 
\begin{align}\label{L2StableProof}
\TwoNorm{\Qint(u)}{\omega_\vnode} 
\lesssim  \left( \TwoNorm{ u}{\omega_{\vnode,1}} + H \TwoNorm{ \nabla u}{\omega_{\vnode,1}}	\right).
\end{align}
\par
For the $H^1$ stability, first noting that 
$\avrg{u}{\omega_{\vnode,1}}=\Qint(\avrg{u}{\omega_{\vnode,1}})$, 
we denote $\bar{u}=u-\avrg{u}{\omega_{\vnode,1}}$.
Thus, from  \eqref{nodeinfty}--\eqref{basisfunctionweighted}, and 
arguments used to obtain \eqref{L2StableProof}, we arrive at 
\begin{align}\label{stability.QI2}
&\TwoNorm{\nabla \Qint(u)}{\omega_{\vnode,1}} 
=\TwoNorm{\nabla \Qint(\bar{u})}{\omega_{\vnode,1}} \nonumber \\
&\qquad\lesssim \sum_{\vnode'\in{\cal N}_{int}(\omega_\vnode)} \left|({\cal P}_{\vnode'} \bar{u})(\vnode')\right| 
  \TwoNorm{ \nabla \lambda_{\vnode'}}{\omega_\vnode}
  +\sum_{\vnode'\in{\cal N}_{\Omega}(\omega_\vnode)} 
  \left|({\cal P}^{\Omega}_{\vnode'} \bar{u})(\vnode')\right| \TwoNorm{\nabla\lambda_{\vnode'}}{\omega_{\vnode}} \nonumber \\
&\qquad\lesssim H^{-1} \left( \TwoNorm{ \bar{u}}{\omega_{\vnode,1}} + H  \TwoNorm{ \nabla \bar{u}}{\omega_{\vnode,1}}\right)
\lesssim\TwoNorm{ \nabla u}{\omega_{\vnode,1}},
\end{align}
where for the last inequality we used the weighted Poincar\'{e} inequality from Lemma \ref{poincare2}.
\par
To prove the local $L^2$ approximability we note that for $\bar{u}=u-\avrg{u}{\omega_{\vnode,1}}$, 
using Lemma \ref{poincare2} and \eqref{L2StableProof} we get
\begin{align}\label{L2StableProofderp}
&\TwoNorm{u-\Qint(u)}{\omega_\vnode}
  =\TwoNorm{\bar{u}-\Qint(\bar{u})}{\omega_\vnode} 
  \lesssim \TwoNorm{\bar{u} }{\omega_\vnode}+\TwoNorm{ \Qint(\bar{u})}{\omega_\vnode}\nonumber \\
  &\qquad\lesssim H  \TwoNorm{ \nabla u}{\omega_{\vnode}}+
      \left( \TwoNorm{ \bar{u}}{\omega_{\vnode,1}} + H  \TwoNorm{ \nabla \bar{u}}{\omega_{\vnode,1}}\right)
  \lesssim H\TwoNorm{ \nabla {u}}{\omega_{\vnode,1}}.
\end{align}
Thus, local approximability holds, and result \eqref{stableproj2} trivially holds from $H^1$ stability. 
\par
From arguments in \cite{brown2016multiscale} it follows that $\Qint$ is also a projection. 
\end{proof}
\begin{corollary}\label{stablelemma.fractional.boundary}
Suppose $\vnode\in {\cal N}_{\Omega}$ and denote 
$\partial \omega_{\vnode}:=\omega_{\vnode}\cap \Omega$. 
Then, for all $u\in  \Hdir{\cilt}$,
it holds that
\begin{align}\label{stableproj1.fractional.boundary}
&\norm{u-\Qint(u)}_{L^2(\partial \omega_{\vnode}) } 
   \lesssim  H^{s} \TwoNorm{\nabla  u}{{\omega}_{\vnode,1} }.
\end{align}
\end{corollary}
\begin{proof}
Recall the weighted trace Lemma \ref{traceHlemma},
and use stability and approximability in the interior from Lemma \ref{stablelemma}
as well as the weighted Poincar\'{e} inequality from Lemma \ref{poincare2} to deduce
\begin{align*} 
  \norm{u-\Qint({u})}_{L^2(\partial \omega_\vnode)} 
  &\lesssim H^{s-1} \TwoNorm{  u-\Qint({u})}{\omega_{\vnode,1} } +H^{s   }  \TwoNorm{  \nabla u}{ \omega_{\vnode,1} }
  \lesssim H^{s}  \TwoNorm{  \nabla u}{ \omega_{\vnode,1} }. \qedhere
\end{align*}
\end{proof}	

\section{Numerical Homogenization}\label{MSmethod}
We will now construct the multiscale approximation space to handle the oscillations created by the 
heterogeneities in the coefficient of the Caffarelli-Silvestre extension problem  \cite{caffarelli2007extension}.
The main ideas of this splitting can be found in 
\cite{brown2016multiscale,Henning.Morgenstern.Peterseim:2014,MP11} 
and references therein. In our computational approach we will for simplicity only consider the 
truncated cylinder $\cilt$ in what follows due to the exponential convergence of the truncated problem to the 
infinite cylinder problem on $\cil$.

\subsection{Multiscale Method}
In this section we construct the multiscale approximation. 
The main ideas of the splitting into a fine-scale and a coarse-scale space 
can be found in \cite{Henning.Morgenstern.Peterseim:2014,MP11} and 
references therein. As noted before the coarse mesh space restricted to $\cilt$ can not resolve the features of 
the microstructure and these fine-scale features must be captured in the multiscale basis. We begin by 
constructing fine-scale spaces.
\par
We define the kernel of the quasi-interpolation operator \eqref{proj} to be 
\begin{align*}
{V}^f=\{v\in  \Hdir{\cilt}  \;| \; \Qint v=0\},
\end{align*}
where $\Qint$ is defined by \eqref{proj}. 
This space will capture the small scale features not resolved by ${V}_H$. 
We define the fine-scale projection $Q_{\cilt}  :{V}_H\to {V}^f$ 
to be the operator such that for $v_H\in {V}_H$ 
we compute $Q_{\cilt} (v)\in {V}^f $ as
\begin{align}\label{corrector}
  \int_{\cilt}B(x)\nabla Q_{\cilt} (v_H) \nabla w\, y^a \,d\x=\int_{\cilt}B(x) \nabla v_H \nabla w\, y^a \,d\x
  \quad\text{for all } w\in {V}^f.
\end{align}
This projection gives an orthogonal splitting $\Hdir{\cilt}={V}^{ms}_{H}\oplus {V}^f$ with the modified coarse 
space 
\[
{V}^{ms}_{H}=({V}_{H}-Q_{\cilt} ( {V}_{H})).
\]
We can decompose any $u\in \Hdir{\cilt}$ as $u=u^{ms}+u^f$ with 
$\int_{\cilt} B(x) \nabla u^{ms} \nabla u^f\, y^a \,d\x=0$.
This  modified coarse space is referred to as the \emph{ideal} multiscale space. The multiscale Galerkin 
approximation  $u_{H}^{ms}\in {V}_{H}^{ms}$ satisfies 
\begin{align}\label{msvarform}
  \int_{\cilt}B(x) \nabla u_{H}^{ms} \nabla v\, y^a \,d\x
  =\int_{\Omega }c_{s} f(x) v(x,0) \,dx 
  \quad\text{ for all } v\in {V}^{ms}_{H} .
\end{align}
\par
The  issue with constructing the solution to \eqref{msvarform} is that the computation of the corrector is global.  
However, it has been shown that the corrector decays exponentially. 
Therefore, we define the localized fine-scale space to be the fine-scale space extended by zero outside the patch, that is 
\[
{V}^f({\omega}_{\vnode,k})=\{v\in {V}^f |\text{   } v|_{\cilt\backslash {\omega}_{\vnode ,k}}=0\}.
\]
We let  for some $\vnode \in {\cal N}_{dof}$ and $k\in \mathbb{N}$ the localized  corrector operator 
$Q_{\vnode,k}: {V}_{H}\to {V}^f({\omega}_{\vnode,k})$
be defined such that given a $u_{H}\in {V}_{H}$
\begin{align}\label{Qcorrector}
  \int_{{\omega}_{\vnode,k}} B(x) \nabla Q_{\vnode ,k}(u_{H}) \nabla w\, y^a \,d\x
  =\int_{{\omega}_{\vnode}}B(x)\hat{\lambda}_{\vnode}\nabla u_{H} \nabla w\, y^a \,d\x
  \quad\text{for all } 
  w \in {V}^f({\omega}_{\vnode ,k}),
\end{align}
where $\hat{\lambda}_{\vnode}=\frac{{\lambda}_{\vnode}}{\sum_{\vnode'\in {\cal N}_{dof }}\lambda_{\vnode'}}$ 
is  augmented so that the collection $\{ \hat{\lambda}_{\vnode }\}_{\vnode\in {\cal N}_{dof}}$ is a partition of 
unity. As in \cite{brown2016multiscale}, this is augmented because the Dirichlet condition makes the standard 
basis not a partition of unity near the boundary.
We denote the global truncated corrector operator as
\begin{align}\label{Qcorrectorbasis.global}
  Q_{k}(u_{H})=\sum_{\vnode \in {\cal N}_{dof}}Q_{\vnode,k}(u_{H}).
\end{align}
 With this  notation, we   write the truncated multiscale space as
\[
  {V}^{ms}_{H,k}=\text{span}\{u_{H}-Q_{k}(u_{H}) | u_{H}\in {V}_{H}  \}.
\]
Moreover, note also that for sufficiently large $k$, we recover the full domain and obtain 
the ideal corrector, denoted $Q_{\cilt }$, with functions of global support from \eqref{corrector}.
The corresponding multiscale approximation to \eqref{varform} is: find $ u_{H,k}^{ms}\in {V}^{ms}_{H,k}$ 
such that
\begin{align}\label{localmsvarform}
  \int_{\cilt}B(x) \nabla u_{H,k}^{ms} \nabla v\, y^a \,d\x=\int_{\Omega}c_s f(x) v(x,0) \,dx 
  \quad\text{for all } v\in {V}^{ms}_{H,k}.
\end{align}

\subsection{Error Analysis}
In this section we present the error introduced by using  \eqref{msvarform} on the global domain to compute 
the solution to \eqref{varform}. 
Then, we show how localization effects the error when we use  \eqref{localmsvarform} on truncated domains 
to compute the same solution.
We also show that, supposing more smoothness in the initial data, and augmenting the 
quasi-interpolation operator to have a global orthogonality condition on $\Omega$, 
that we may obtain a better rate of convergence.

\subsubsection{Error with Global Support}
\begin{theorem}\label{errorglobal}
Suppose that $u\in \Hdir{\cilt}$ satisfies \eqref{varform} and that $u_{H}^{ms}\in {V}_{H}^{ms}$ satisfies 
\eqref{msvarform}.  Suppose the data is such that $f\in L^2(\Omega)$. Then, we have the following error 
estimate
\begin{align}\label{eq:errorglobal}
  \TwoNorm{\nabla u-\nabla u_{H}^{ms}}{\cilt}\lesssim H^{s} \norm{f}_{L^2(\Omega)}.
\end{align}
\end{theorem}
\begin{proof}
We  utilize the local stability property of $\Qint$ from Lemma \ref{stablelemma}, and in particular the trace 
estimate of Corollary \ref{stablelemma.fractional.boundary}. From the orthogonal splitting of the spaces it is 
clear that $ u- u_{H}^{ms}=u^f \in {V}^f$ and $\Qint(u^f)=0$. Thus,  utilizing Galerkin orthogonality,  taking the 
test function in the variational form to be $v=u^f=u-u_H^{ms}$ we have
\begin{align*}
  \TwoNorm{\nabla u-\nabla u_{H}^{ms}}{\cilt}^2
  &\lesssim \int_{\cilt}B(x) |\nabla u^f|^2\, y^a \,d\x=\int_{\Omega }c_{s} f(x) (u^f-\Qint(u^f)) \,dx\\
  &\lesssim \norm{f}_{L^2(\Omega)}\norm{ u^f-\Qint(u^f)}_{L^2(\Omega)} 
  \lesssim  H^{s}  \norm{f}_{L^2(\Omega)}\TwoNorm{\nabla u^f}{\cilt},
\end{align*}
where we used the approximation property \eqref{stableproj1.fractional.boundary}.
Dividing the last $\TwoNorm{\nabla u^f}{\cilt}$ term yields the result.
\end{proof}
 \begin{remark}
Note that we obtain the expected convergence rate of $H^s$ for the 
fractional Laplacian type problems on quasi-uniform meshes.
Further,  we do not need to utilize second order derivatives of $u$ as in 
the analysis of \cite[Section 5]{nochetto2015pde}. In that setting, the term $\norm{u_{yy}}_{L^2(\cilt,y^\beta) }$ 
yielded a convergence rate of $C_{\eps} H^{s-\eps},$ for all $\eps>0$, with the constant blowing up as 
$\eps \to 0$. 
However, the sub-grid fine $h$ standard finite elements may suffer from these effects. Here we focus 
merely on the error accumulated from the coarse-grid.
\end{remark}

\subsubsection{Error with Localization}\label{localizationsections}
In this  section,  we discuss the error due to the truncation of the corrector problems
to patches of $k$ layers.
The key lemma needed is the following lemma, that gives
the decay in the error as the  truncated corrector approaches the ideal corrector of global 
support in the weighted Sobolev norm.
\begin{lemma}\label{localglobal.derp}
	Let $u_{H}\in {V}_{H}$, let $Q_{k}$ be constructed from \eqref{Qcorrector} and 
\eqref{Qcorrectorbasis.global}, and $Q_{\cilt}$ defined to be the ideal corrector without truncation in 
\eqref{corrector}, then for some $\theta\in (0,1)$
\begin{align}
  \TwoNorm{\nabla( Q_{\cilt}(u_{H})-Q_{k}(u_{H}))  }{\cilt}
  \lesssim    k^{\frac{d}{2}}  \theta^{k}  \TwoNorm{ \nabla  u_{H}  }{\cilt}.
\end{align}
\end{lemma}
\begin{proof}
See Appendix \ref{truncproofsection}.
\end{proof}
\begin{theorem}\label{errorlocal}
Suppose that $u\in \Hdir{\cilt}$ satisfies \eqref{varform} and that $u_{H,k}^{ms}\in {V}_{H,k}^{ms}$, with 
local correctors  calculated from \eqref{Qcorrector},
satisfies \eqref{localmsvarform}. 
Suppose $f\in L^2(\Omega).$ Then, we have the following error estimate for some $\theta \in (0,1)$
\begin{align}\label{eq:errorlocal1}
  \TwoNorm{\nabla u-\nabla u_{H,k}^{ms}}{\cilt}
  \lesssim \left(  H^{s}+ k^{\frac{d}{2}}  \theta^{k}\right)\norm{f}_{L^2(\Omega)}.
\end{align}
\end{theorem}
\begin{proof}
We let $u^{ms}_{H}=u_{H}-Q_{\cilt}(u_{H})$ be the ideal global multiscale solution 
satisfying \eqref{msvarform}, and $u^{ms}_{H,k}=u_{H,k}-Q_{k}(u_{H,k})$ 
be the corresponding truncated solution to \eqref{localmsvarform}. 
Then, by Galerkin approximations being minimal in energy norm we have
\begin{align*}
  \TwoNorm{\nabla u-\nabla (u_{H,k}-Q_{k}(u_{H,k}))}{\cilt}
  \lesssim \TwoNorm{\nabla u-\nabla( u_{H}-Q_{k}(u_{H}) )}{\cilt}.
\end{align*}
Using this fact and Theorem \ref{errorglobal} and Lemma \ref{localglobal.derp} we have
\begin{align*}
  \TwoNorm{\nabla u-\nabla u_{H,k}^{ms}}{\cilt}
  &\leq\TwoNorm{\nabla u-\nabla( u_{H}-Q_{\cilt}(u_{H})+Q_{\cilt}(u_{H})-Q_{k}(u_H))}{\cilt} \\
  &\leq\TwoNorm{\nabla u-\nabla u_{H}^{ms}}{\cilt}+\TwoNorm{\nabla( Q_{\cilt}(u_{H})-Q_{k}(u_{H}))  }{\cilt}\\
  &\lesssim H^{s}\norm{f}_{L^2(\Omega)}+ k^{\frac{d}{2}}  \theta^{k} \TwoNorm{ \nabla u_{H}  }{\cilt}.
\end{align*}
In addition note that, by construction, 
$\Qint(u_{H}^{ms})=\Qint(u_{}^{})$.
Thus, using  local stability \eqref{stableproj0} and \emph{a-priori} bounds from \eqref{msvarform}, 
obtained via the trace inequality in Lemma \ref{CanonicalTrace}, we have
\begin{align*}
  \TwoNorm{ \nabla u^{}_{H} }{\cilt} 
  &\lesssim\TwoNorm{ \nabla \Qint(u_{H}^{ms}) }{\cilt}
  \lesssim \TwoNorm{ \nabla u_{H}^{ms} }{\cilt}
  \lesssim\norm{f}_{L^2(\Omega)}. \qedhere
\end{align*}
\end{proof}

\subsubsection{Error with $L^2$ Projection on $\Omega$}
By augmenting our quasi-interpolation \eqref{proj} on the boundary $\Omega$ we may obtain a better order of 
convergence given sufficiently smooth data $f$. We instead define the quasi-interpolation 
\begin{align}\label{proj.L2.omega}
  \Qint^{L^2} u(\x)=\sum_{\vnode\in{\cal N}_{int}} ({\cal P}_{\vnode} u)(\vnode) \lambda_{\vnode}(\x)+
  \sum_{\vnode\in{\cal N}_{\Omega}} (\Pi_{\Omega}^{L^2} u)(\vnode)\lambda_{\vnode}(\x),
\end{align}
where $\Pi_{\Omega}^{L^2} :L^2(\Omega)\to V_{H}|_{\Omega}$ is the (global on $\Omega$) $L^2$ projection
\begin{align*}
  \int_{\Omega}(\Pi_{\Omega}^{L^2}u) v_{H}\,  dx 
  =\int_{\Omega} u v_{H}\,  dx 
  \quad\text{for all } v_{H}\in V_{H}|_{\Omega}.
\end{align*}
From this we see that by construction for $f_{H}\in  V_{H}|_{\Omega}$ we have 
\begin{align}\label{orthogonalL2}
  \int_{\Omega} f_{H} \tilde{v} \,dx=0 
  \quad\text{for} \quad
  \tilde{v}\in \text{ker}\left( \Qint^{L^2}\right).
\end{align}
\begin{remark}
We suppose $\Qint^{L^2}$ given by \eqref{proj.L2.omega} satisfies the stability relations  in Lemma 
\ref{stablelemma} and Corollary \ref{stablelemma.fractional.boundary} as  similar arguments provided in those 
proofs will hold.
\end{remark}
\begin{theorem}\label{errorglobal.L2}
Suppose that $u\in \Hdir{\cilt}$ satisfies \eqref{varform} and that $u_{H}^{ms}\in {V}_{H}^{ms}$ satisfies 
\eqref{msvarform}, where the spaces are constructed using $\Qint^{L^2}$ from \eqref{proj.L2.omega}. We 
suppose  the additional  regularity $f\in \mathbb{H}^{1-s}$.
Then, we have the following error estimate
\begin{align}\label{eq:errorglobal.L2}
  \TwoNorm{\nabla u-\nabla u_{H}^{ms}}{\cilt}
  \lesssim H^{s} \inf_{f_{H}\in V_{H}|_{\Omega}}\left(\norm{f-f_{H}}_{L^2(\Omega)} \right)
  \lesssim  H^{}\norm{f}_{\mathbb{H}^{1-s}(\Omega)}.
\end{align}
\end{theorem}
\begin{proof}
We  again utilize the local stability property of $\Qint$ from Lemma \ref{stablelemma}, and in particular 
the trace estimate of Corollary \ref{stablelemma.fractional.boundary}.
Thus,  utilizing Galerkin orthogonality, the orthogonality relation \eqref{orthogonalL2},
and taking the test function in the variational form to be 
$v=u^f=u-u_H^{ms}$, we arrive at
\begin{align*}
   \TwoNorm{\nabla u-\nabla u_{H}^{ms}}{\cilt}^2
   &\lesssim \int_{\cilt}B(x) |\nabla u^f|^2\, y^a \,d\x
   =\int_{\Omega }c_{s} f \, (u^f-\Qint^{L^2}(u^f)) \,dx\\
   &=\int_{\Omega }c_{s} (f-f_{H}) (u^f-\Qint^{L^2}(u^f)) \,dx
   \lesssim \norm{f-f_H}_{L^2(\Omega)}\norm{ u^f-\Qint^{L^2}(u^f)}_{L^2(\Omega)} \\
   &\lesssim H^{s}  \norm{f-f_H}_{L^2(\Omega)}\TwoNorm{\nabla u^f}{\cilt}.
\end{align*}
Dividing the last $\TwoNorm{\nabla u^f}{\cilt}$ and using the standard interpolation estimate
\[
  \inf_{f_{H}\in V_{H}|_{\Omega}}\left(\norm{f-f_{H}}_{L^2(\Omega)} \right)\lesssim H^{1-s}\norm{f}_{\mathbb{H}^{1-s}(\Omega)},
\]
yields the result.
\end{proof}
\begin{remark}\label{remark:5.7}
The use of the global $L^2$ projection on $\Omega$ in the construction of the method does 
\emph{not} require  global-on-$\Omega$ computation. In fact, the quasi-interpolation operator 
\eqref{proj.L2.omega} does not need 
to be computed at all. The method solely requires the characterization of its kernel which can be realized via 
local functional constraints associated with coarse nodes, i.e.,
${\cal P}_{\vnode}\left( \cdot\right)(\vnode) =0$,
for all interior coarse nodes ${\cal N}_{int}$, and $(\cdot,\lambda_\vnode)_{\Omega}=0$
for nodes ${\cal N}_{\Omega}$, on $\Omega$.
\end{remark}
\begin{remark}
A similar truncation argument from Section \ref{localizationsections}, can be shown to also hold in this setting.
\end{remark}

\section{Numerical Examples}\label{numerics}
In this section we present some  numerical examples 
for $\Omega\subset \mathbb{R}^2$ or  ${\cal C} \subset \mathbb{R}^2 \times \mathbb{R}_{+}$, to illustrate the 
convergence behavior
of the multiscale method.
In particular, we observe higher order 
convergence for a simple generic analytic example
even for local boundary projections onto $\Omega$, using ${\cal P}_{\vnode}^\Omega$,
as indicated by Remark~\ref{remark:5.7}.
However, we demonstrate that with an heterogeneous coefficient this is not the case.
We will compare the multiscale approximation $u_H^{ms}$ to a fine-scale approximation $u_h$,
by replacing $u$ by $u_h$ in the theoretical results.
For $s<0.5$, we truncate  the domain in the extension direction at $T=1$, and for 
$s>0.5$ at $T=1.5$. 
These truncation lengths have been empirically found to be sufficient for the fine grid approximations.
For numerical efficiency we truncate the computations of the correctors to a local element patch
of size $k=2$ from the truncation estimate  Lemma \ref{localglobal.derp}.
In all experiments we use linear Lagrange finite elements. We will give two examples, 
one with a homogeneous and one 
with a heterogeneous coefficient and test $s$-values above and below the critical $s=0.5$ value. 
\subsection{Analytic example}
\begin{figure}
\centering
\includegraphics[width = 0.5\textwidth]{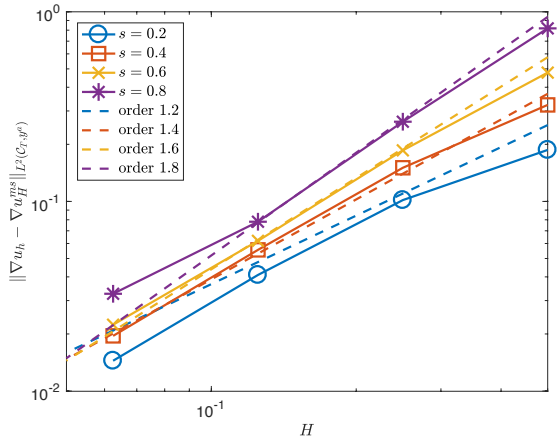}
\caption{Convergence for the analytic example.}
\label{fig:ex1:convergence}
\end{figure}
We take the analytic example from \cite[Section 6.1]{nochetto2015pde} 
with $\Omega=(0,1)^2$ and so $\cilt= (0,1)^2 \times (0,T)$ with the  forcing  
\[
f(x_1,x_2) = (2\pi^2)^{s}\sin(\pi x_1)\sin(\pi x_2).
\]
The exact solution on $\Omega$ is then given by
$u(x_1,x_2) = \sin(\pi x_1)\sin(\pi x_2)$ and the exact solution on the extended domain $\mathcal{C}_T$ by
\[
	u(x_1,x_2,y) = \frac{2^{1-s}}{\Gamma(s)} (2\pi^2)^{s/2} \sin(\pi x_1)\sin(\pi x_2) y^s K_s(\sqrt{2}\pi y),
\]
where $K_s$ denotes the modified Bessel function of the second kind.
\par
Note that $f$ is smooth in this example, hence the estimate in Theorem~\ref{errorglobal.L2}
can be improved to
\[
\| \nabla u_h - \nabla u_H^{ms}\|_{L^2(\mathcal{C}_T, y^a)} \lesssim H^{1+s}\norm{f}_{H^1(\Omega)}.
\]
Here, Figure~\ref{fig:ex1:convergence} shows the convergence of the error
$\| \nabla u_h - \nabla u_H^{ms}\|_{L^2(\mathcal{C}_T, y^a)}$ for $H=2^{-1},\ldots,2^{-4}$ and
$h=2^{-6}$. As predicted by the theory, we observe numerical convergence 
close to $\mathcal{O}(H^{1+s})$ for $s=0.2,0.4,0.6$ and $s=0.8$.
Note that in this particular example we get improved convergence rates despite the fact
that we used local projections ${\cal P}_{\vnode}^{\Omega}$. This indicates that the sum of the local 
projections are close to the
$\Omega$-global $L^2$ projection in this simple example.
\subsection{Heterogeneous example}
\begin{figure}
\centering
\subfloat[][]{
\includegraphics[width = 0.5\textwidth]{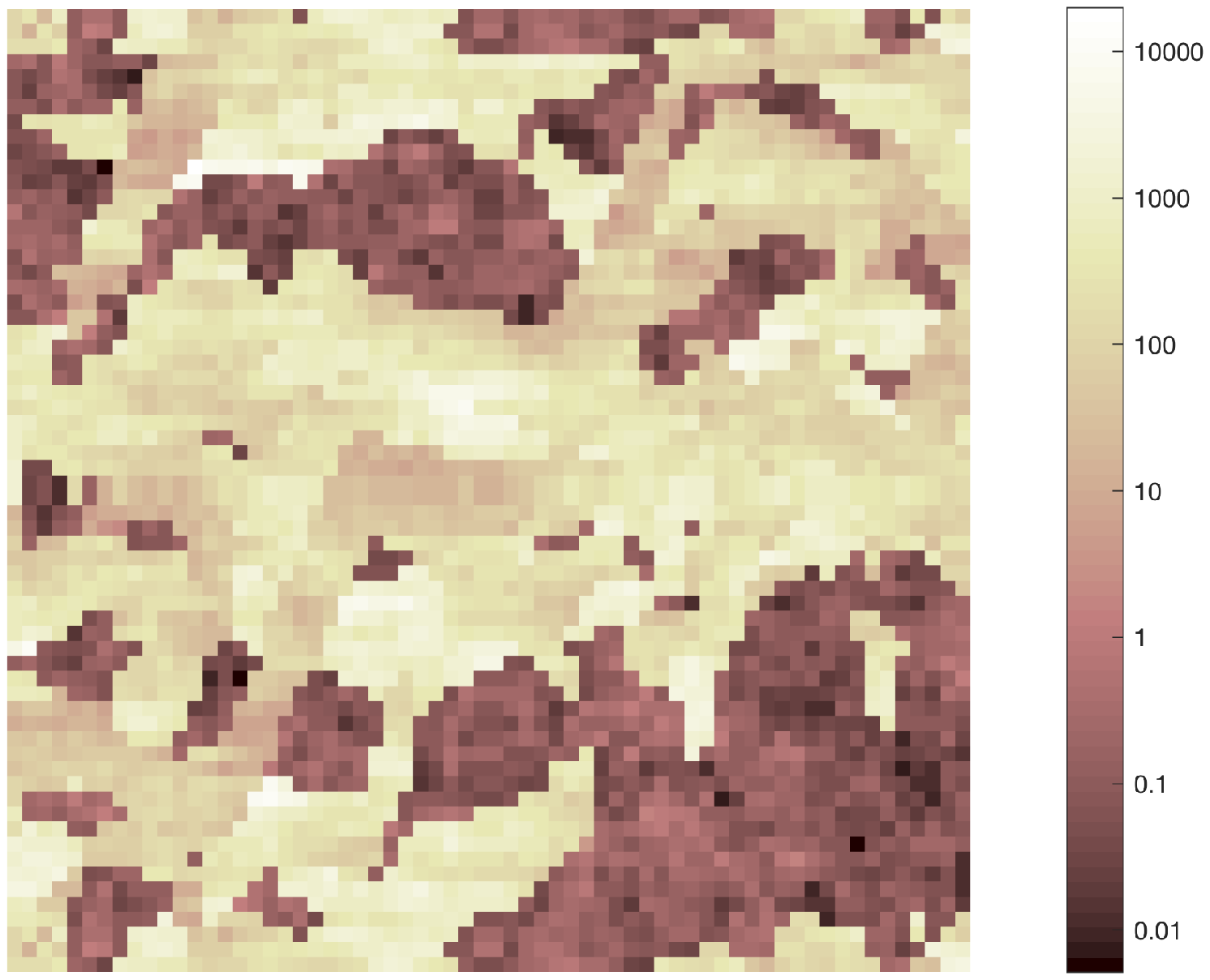}
\label{fig:permeability}
}

\subfloat[][]{
\includegraphics[width = 0.5\textwidth]{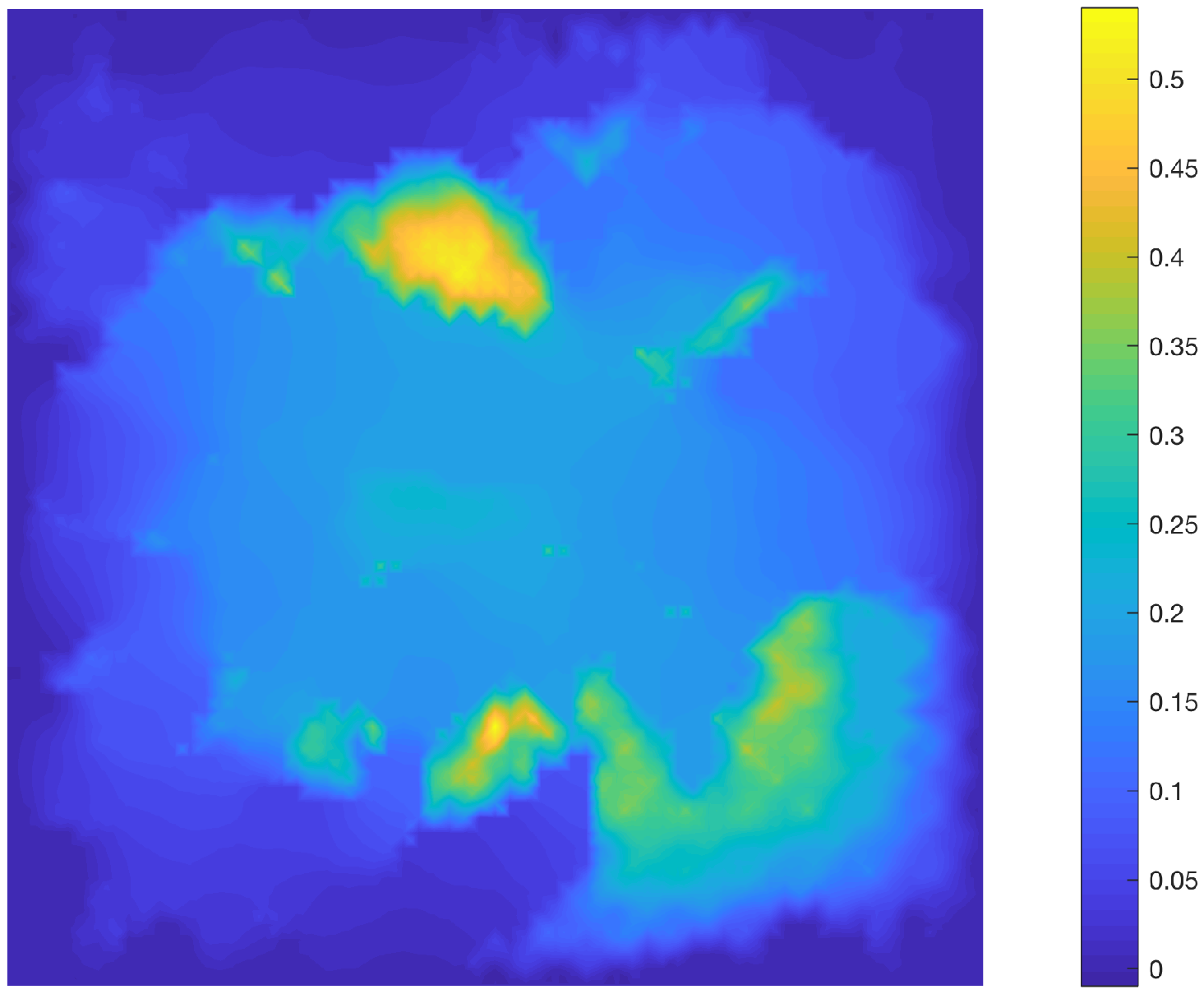}
\label{fig:speSolution}
}
\subfloat[][]{
\includegraphics[width = 0.5\textwidth]{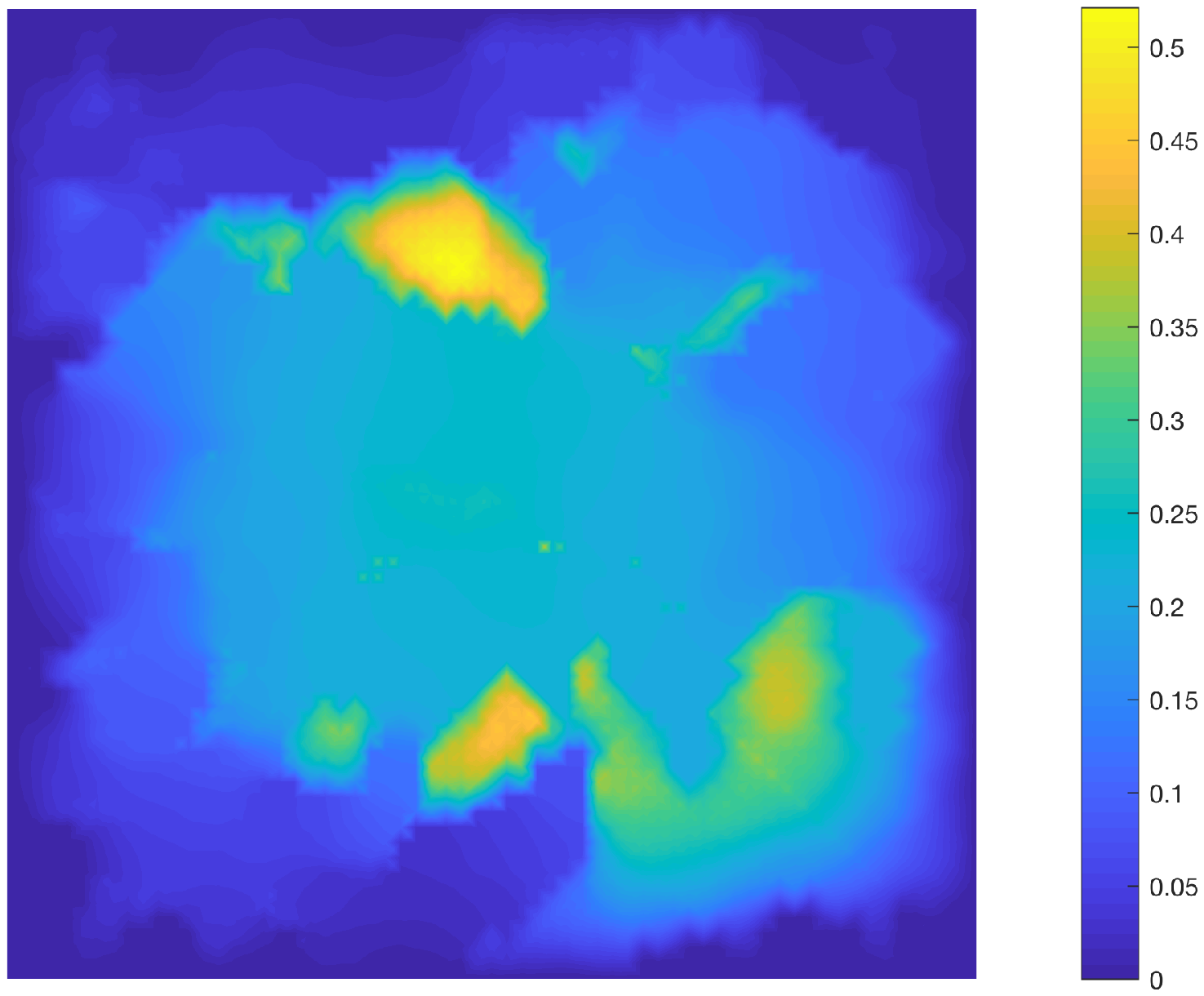}
\label{fig:speFineSolution}
}
\caption{ \protect\subref{fig:permeability} Logarithm of the chosen permeability, 
\protect\subref{fig:speSolution} discrete multiscale solution for $s=0.2$ and $k=2$, and
\protect\subref{fig:speFineSolution} fine scale approximation for $s=0.2$ and $h=2^{-9}$ in the heterogeneous example.
}
\label{fig:permeability:speSolution}
\end{figure}
\begin{figure}
\centering
\includegraphics[width = 0.5\textwidth]{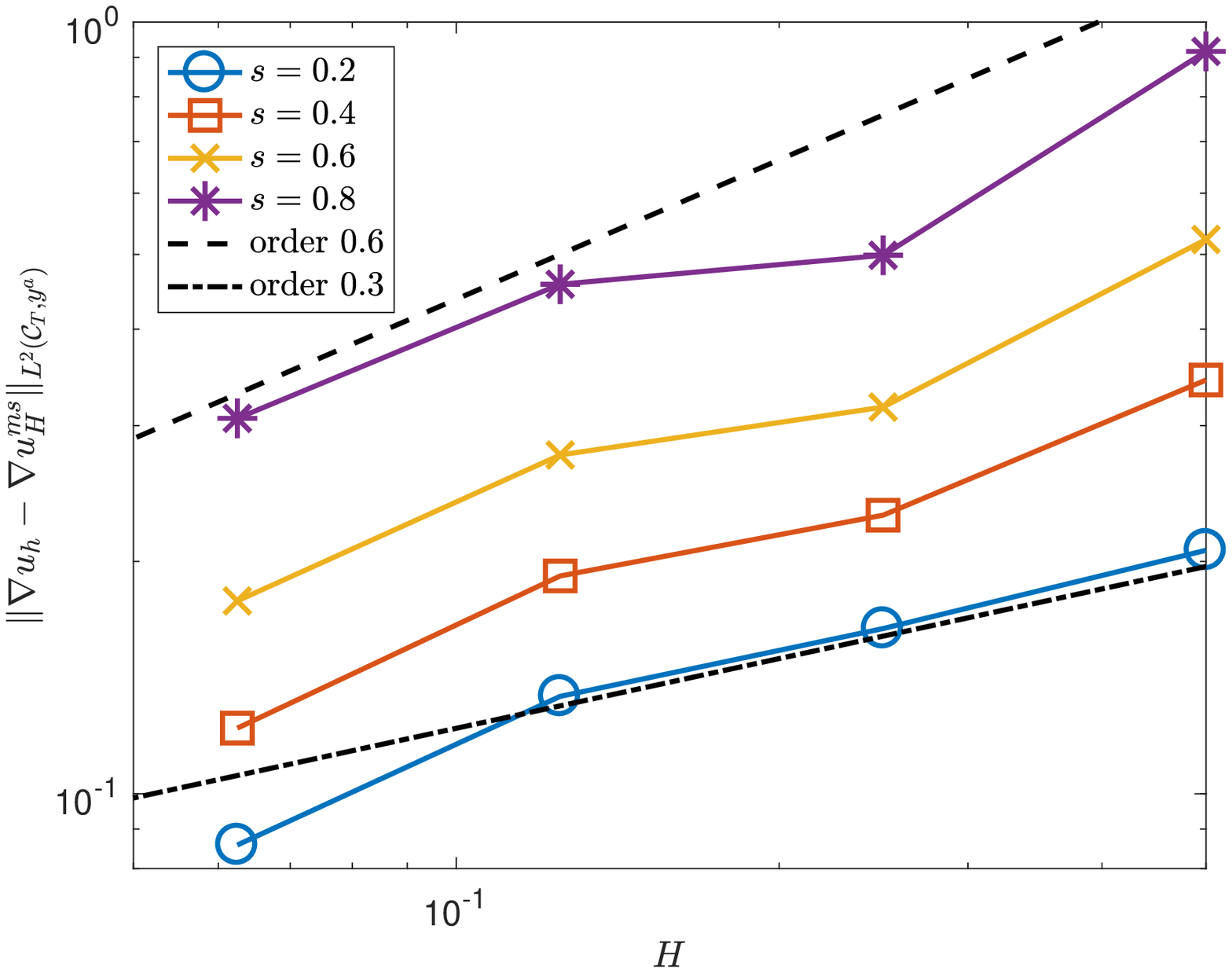}
\caption{Convergence for the heterogeneous example.}
\label{fig:ex2:convergence}
\end{figure}
In this example we choose again $\Omega=(0,1)^2$   so that  $\cilt= (0,1)^2 \times (0,T)$, and
$f(x_1,x_2) = (2\pi^2)^s\sin(\pi x_1)\sin(\pi x_2)$.
However,  we chose a non-constant diffusion coefficient that varies on the fine scale between
$5\cdot 10^{-3}$ and $2\cdot 10^4$. The values are taken from the SPE10 benchmark problem.
The logarithm of the chosen values is displayed in Figure~\ref{fig:permeability}.
A discrete multiscale solution  for $s=0.2$ is displayed in Figure~\ref{fig:speSolution},
and the fine scale approximation for $s=0.2$ and $h=2^{-6}$ is displayed in Figure~\ref{fig:speFineSolution}.
Comparing Figures \ref{fig:speSolution} and \ref{fig:speFineSolution} one can observe that the 
LOD method captures the fine scale features of the solution very well.
We shall emphasize that the theory of localization (Lemma~\ref{localglobal.derp} and Appendix~\ref{truncproofsection}) 
does not allow meaningful predictions on the performance of the multiscale method in the present regime of 
very high contrast. Still, the experimental results are promising. This has also been observed before for 
high-contrast local PDEs in \cite{Peterseim.Scheichl:2016,brown2016multiscale}. The theory therein also indicates 
that the success of numerical homogenization may depend on the geometric properties of the diffusion  
coefficient and its phases relative to the coarse mesh. In particular, a non-monotonic behavior of the error 
may occur depending on the relative position of coarse nodes and high and low permeability regions of the 
medium.
In Figure~\ref{fig:ex2:convergence},  the convergence of the error
$\| \nabla u_h - \nabla u_H^{ms}\|_{L^2(\mathcal{C}_T, y^a)}$ for $H=2^{-1},\ldots,2^{-4}$ and
$h=2^{-6}$ is shown.
Because of the heterogeneous coefficient, we cannot expect 
the local projections ${\cal P}_{\vnode}^{\Omega}$ to be close
to the $\Omega$-global $L^2$ projection,
hence we expect convergence rates of $\mathcal{O}(H^{s})$ from Theorem~\ref{errorlocal}.
As shown in Figure~\ref{fig:ex2:convergence} we indeed observe convergence rates
in the range of $\mathcal{O}(H^{s})$ despite the high contrast
of the diffusion coefficient and the small truncated patches of the corrector problems. 
For $s=0.2$ we even see some minor improved convergence of $\mathcal{O}(H^{0.3})$
due to the boundary projections, 
while for $s=0.8$ the convergence of $\mathcal{O}(H^{0.6})$ 
is lower due to the rough truncation of the corrector problems at layer $k=2$.
Note that the error of the fine-grid solution is probably much higher, so that
higher computational costs for larger $k$ are not justified.
The convergence for $s=0.4$ and $s=0.6$ are in between those values and therefore closer to the predicted 
value $\mathcal{O}(H^{s})$.

\section{Conclusion}
In this paper, we developed a multiscale method for heterogeneous fractional Laplacians. The method utilized 
a localization of multiscale correctors to obtain an efficient numerical scheme with
optimal rates of convergence for the coarse-grid. We 
developed this method in the context of weighted Sobolev spaces to be applied to the extended domain 
problem of the fractional Laplacian where the coefficient of the extension has a singular/degenerate value. To 
this end, we constructed a quasi-interpolation that utilizes averages on $d$ and $d+1$ dimensional subsets so 
that the critical boundary is better resolved. We proved local stability and approximability of this operator in 
weighted Sobolev spaces. We then proved the error estimates and truncation arguments in this weighted setting. To 
confirm our theoretical results we gave two numerical experiments with various fractional orders $s$.

\section{Acknowledgments}
%
%
The second author has been funded by the Austrian Science Fund (FWF) through the
project P 29197-N32.
%
Main parts of this paper were written while the authors enjoyed
the kind hospitality of the Hausdorff Institute for Mathematics in Bonn during the trimester program on 
multiscale problems in 2017.

\appendix
\section{Truncation Proofs}\label{truncproofsection}
Now we will prove and state the auxiliary lemmas used to prove the localized error estimate in Lemma 
\ref{localglobal.derp} and Theorem \ref{errorlocal}.  These proofs are largely based on the works 
\cite{Henning.Morgenstern.Peterseim:2014,MP11} and references therein. There are a few interesting nuances 
with respect to the weighted inverse and Poincar\'{e} inequalities, Muckenhoupt constant bounds, and the  
Caccioppoli inequality Lemma \ref{Caccioppoli}.
\par
We begin with some notation.
For $\vnode,\vnode' \in{\cal N}_{dof}$ and $l,k\in \mathbb{N}$ 
and $m=0,1,\cdots,$ with $k\geq l\geq 2$ we have the quasi-inclusion property:
\begin{align}\label{quasiinclusion}
\text{if}\quad {\omega}_{\vnode',m+1}\cap \left({\omega}_{\vnode,k}\backslash {\omega}_{\vnode,l} \right)\neq \emptyset, 
\quad\text{then}\quad
{\omega}_{\vnode',1}\subset\left({\omega}_{\vnode,k+m+1}\backslash {\omega}_{\vnode,l-m-1} \right).
\end{align}
We will use the cutoff functions defined in \cite{Henning.Morgenstern.Peterseim:2014}. 
For $\vnode\in {\cal N}_{dof}$ and $k>l\in \mathbb{N}$, let $\eta^{k,l}_{\vnode}:\cilt \to [0,1]$ be a continuous 
weakly differentiable function so that 
\begin{subequations}\label{cutoff1}
\begin{align}
  \left(   \eta^{k,l}_{\vnode}  \right)|_{{\omega}_{\vnode,k-l}}&=0,\\
  \left(   \eta^{k,l}_{\vnode}  \right)|_{\cilt \backslash {\omega}_{\vnode,k}}&=1,\\
  \forall T\in {\cal T}_{\cilt}, \norm{\nabla \eta^{k,l}_\vnode}_{L^{\infty}(T)}&\leq C_{co}\frac{1}{l H_{}}, 
\end{align}
\end{subequations}
where $C_{co}$ is only dependent on the shape regularity of the mesh ${\cal T}_{\cilt}$.
We  choose here the cutoff function as in \cite{MP11}, where we choose a function $\eta^{k,l}_{\vnode}$, in the 
space of $\mathbb{P}_{1}$ Lagrange finite elements over ${\cal  T}_{\cilt}$, such that 
\begin{align*}
  \eta^{k,l}_{\vnode}(\vnode')&=0 \text{ for all } \vnode' \in {\cal N}_{dof}\cap \omega_{\vnode,k-l},\nl
  \eta^{k,l}_{\vnode}(\vnode')&=1 \text{ for all } \vnode'\in {\cal N}_{dof}\cap (\cilt\backslash \omega_{\vnode,k}),\nl
  \eta^{k,l}_{\vnode}(\vnode')&=\frac{j}{l} \text{ for all } \vnode'\in {\cal N}_{dof}\cap \omega_{\vnode,k-l+j}, j=0,1, \dots ,l.
\end{align*}
We will now prove the quasi-invariance of the fine-scale functions under multiplication by 
cutoff functions in weighted Sobolev spaces.
 \begin{lemma}\label{qi}
Let $k>l\in \mathbb{N}$ and  $\vnode \in {\cal N}_{dof}$. 
Suppose that $w\in {V}^f$, then we have the estimate
\begin{align*}
 \TwoNorm{\nabla \Qint(\eta_{\vnode}^{k,l} w)   }{\cilt}
 \lesssim l^{-1}\TwoNorm{ \nabla w }{ {\omega}_{\vnode,k+1}\backslash {\omega}_{\vnode,k-l-1} }.
\end{align*}
\end{lemma}
\begin{proof}
Fix $\vnode$ and $k$,  denote the average  as 
$\avrg{\eta_{\vnode}^{k,l}}{{\omega}_{\vnode',1} }
=\frac{1}{|{\omega}_{\vnode',1}|} \int_{{\omega}_{\vnode',1}} \eta_{\vnode }^{k,l}\,d\x.$
For an estimate on a single patch ${\omega}_{\vnode'}$, using the fact that $\Qint(w)=0$ and the 
stability \eqref{stableproj0}, we have 
\begin{align*}
  \TwoNorm{\nabla \Qint(\eta_{\vnode}^{k,l}w )}{{\omega}_{\vnode'}}
  &=\TwoNorm{\nabla \Qint((\eta_{\vnode}^{k,l}-\avrg{\eta_{\vnode}^{k,l}}{{\omega}_{\vnode',1}}) w)}{{\omega}_{\vnode'}}
  \lesssim \TwoNorm{\nabla ( (\eta_{\vnode}^{k,l}-\avrg{\eta_{\vnode}^{k,l}}{{\omega}_{\vnode',1}}) w)}{{\omega}_{\vnode',1}}\\
  &\lesssim \left( \TwoNorm{(\eta_{\vnode}^{k,l}-\avrg{\eta_{\vnode}^{k,l}}{{\omega}_{\vnode',1}})\nabla w }{ {\omega}_{\vnode',1} }    
    +   \TwoNorm{\nabla  \eta^{k,l}_{\vnode} (w-\Qint(w)) }{ {\omega}_{\vnode',1}}  \right).
\end{align*}
Summing over all $\vnode'\in {\cal N}_{dof}$ 
and using the quasi-inclusion property \eqref{quasiinclusion} yields 
\begin{align}
 \TwoNorm{\nabla \Qint(\eta_{\vnode}^{k,l} w)}{\cilt}^2 
 &\lesssim  \sum_{{\omega}_{\vnode'}\subset{\omega}_{\vnode,k+1}\backslash {\omega}_{\vnode,k-l-1} } 
  \TwoNorm{ (\eta_{\vnode}^{k,l}-\avrg{\eta_{\vnode}^{k,l}}{{\omega}_{\vnode',1}})\nabla w }{ {\omega}_{\vnode',1} }^2
  \nl \label{summingover}
 &+  \sum_{{\omega}_{\vnode'}\subset{\omega}_{\vnode,k+1}\backslash {\omega}_{\vnode,k-l-1} }     
  \TwoNorm{\nabla  \eta^{k,l}_{\vnode}( w-\Qint(w)) }{ {\omega}_{\vnode',1}}^2.
\end{align}
Note that we used that $\nabla \eta_{\vnode}^{k,l}\neq 0$ only in 
${\omega}_{\vnode,k}\backslash{\omega}_{\vnode,k-l} $ 
and 
$(\eta_{\vnode}^{k,l}-\avrg{\eta_{\vnode}^{k,l}}{{\omega}_{\vnode',1}})\neq0$ 
only if ${\omega}_{\vnode'}$ intersects ${\omega}_{\vnode,k}\backslash{\omega}_{\vnode,k-l} $,
hence we obtained the slightly better bound.
\par
We now denote 
$\mu_{\vnode}^{k,l}=\eta_{\vnode}^{k,l}-\avrg{\eta_{\vnode}^{k,l}}{{\omega}_{\vnode',1}}$, 
and let ${T}$ be a  simplex in ${{\omega}_{\vnode',1}}$
such that the supremum 
$\norm{\mu_{\vnode}^{k,l}}_{L^{\infty}({\omega}_{\vnode',1})}$ is obtained.  
On ${T}$, $\mu_{\vnode}^{k,l}$ is an affine function, 
using the  fact that $\eta_{\vnode}^{k,l}$ is taken to be $\mathbb{P}_{1}$, we have by using the inverse 
estimate \eqref{inftytoL2} that 
\begin{align*}
\norm{\mu_{\vnode}^{k,l}}_{L^{\infty}({\omega}_{\vnode',1})}
=\norm{\mu_{\vnode}^{k,l}}_{L^\infty(T ) }
&\lesssim  |T|^{-1}\norm{y^{-\frac{a}{2}}}_{L^2(T) }  \norm{\mu_{\vnode}^{k,l}}_{L^2(T,y^a ) }.
\end{align*}	
Using the  above estimate and the weighted Poincar\'{e} inequality, we see that
\begin{align}\label{Clip}
\norm{\eta_{\vnode}^{k,l}-\avrg{\eta_{\vnode}^{k,l}}{{\omega}_{\vnode',1}}}_{L^{\infty}({\omega}_{\vnode',1})}
&\lesssim  |{\omega}_{\vnode',1}|^{-1} \norm{y^{-\frac{a}{2}}}_{L^2({\omega}_{\vnode',1}) } 
    \norm{\eta_{\vnode}^{k,l}-\avrg{\eta_{\vnode}^{k,l}}{{\omega}_{\vnode',1}}}_{L^{2}({\omega}_{\vnode',1},y^a)}\nl
&\lesssim  |{\omega}_{\vnode',1}|^{-1} \norm{y^{-\frac{a}{2}}}_{L^2({\omega}_{\vnode',1}) } 
    H\norm{\nabla \eta_{\vnode}^{k,l}}_{L^{2}({\omega}_{\vnode',1},y^a)}\nl
&\lesssim  |{\omega}_{\vnode',1}|^{-1} \norm{y^{-\frac{a}{2}}}_{L^2({\omega}_{\vnode',1}) } 
    H\sum_{T\in \omega_{\vnode',1}}\norm{\nabla \eta_{\vnode}^{k,l}}_{L^{2}(T,y^a)}\nl
&\lesssim  |{\omega}_{\vnode',1}|^{-1} \norm{y^{-\frac{a}{2}}}_{L^2({\omega}_{\vnode',1}) } 
    H\sum_{T\in \omega_{\vnode',1}}  \norm{y^{\frac{a}{2}}}_{L^2(T) }  \norm{\nabla \eta_{\vnode}^{k,l}}_{L^{\infty}(T)}\nl
&\lesssim \left( |{\omega}_{\vnode',1}|^{-1} \norm{y^{-\frac{a}{2}}}_{L^2({\omega}_{\vnode',1}) } 
    \norm{y^{\frac{a}{2}}}_{L^2({\omega}_{\vnode',1}) } \right) H  \norm{\nabla \eta_{\vnode}^{k,l}}_{L^{\infty}({\omega}_{\vnode',1})}\nl
&\lesssim  C_{2,a}^{\frac{1}{2}}H\norm{\nabla \eta_{\vnode}^{k,l} }_{L^{\infty}({\omega}_{\vnode',1})},
\end{align}
where we used the Muckenhoupt weight bound \eqref{Muckenhouptconstant}, as well as quasi-uniformity of 
the grid.
Returning to \eqref{summingover}, using the above relation on the first term and  the approximation property 
\eqref{stableproj1} on the second term, we obtain
\begin{align*}
\TwoNorm{\nabla \Qint(\eta_{\vnode}^{k,l} w)   }{\cilt}^2 
&\lesssim H^2 \norm{\nabla \eta_{\vnode}^{k,l} }_{L^{\infty}(\cilt)}^2 
   \TwoNorm{ \nabla w }{ {\omega}_{\vnode,k+1}\backslash {\omega}_{\vnode,k-l-1} }^2   \\
 &+  H^2 \norm{\nabla \eta_{\vnode}^{k,l} }_{L^{\infty}(\cilt)}^2  
   \TwoNorm{ \nabla w }{ {\omega}_{\vnode,k+1}\backslash {\omega}_{\vnode,k-l-1} }^2 .
\end{align*}
Finally, we arrive at 
\begin{align*}
  \TwoNorm{\nabla \Qint(\eta_{\vnode}^{k,l} w)   }{\cilt}^2 \lesssim l^{-2}
  \TwoNorm{ \nabla w }{ {\omega}_{\vnode,k+1}\backslash {\omega}_{\vnode,k-l-1} }^2,
\end{align*}
where we used $ \norm{\nabla \eta_{\vnode}^{k,l} }_{L^{\infty}(\cilt)}^2\lesssim 1/(lH)^2$.
\end{proof}
For the weighted Sobolev space, we have the following decay of the fine-scale space.
\begin{lemma}\label{decaylemma}
Fix some $\vnode\in {\cal N}_{dof}$ and let $F\in ({V}^f)'$ be the dual of ${V}^f$ 
satisfying $F(w)=0$ for all $w\in {V}^f(\cilt\backslash {\omega}_{\vnode,1})$.  
Let $u\in {V}^f$ be the solution of 
\begin{align*}
    \int_{\cilt}B(x) \nabla u \nabla w \, y^a \,d\x =F(w) \quad\text{for all } w\in {V}^f,
\end{align*}
then there exists a constant $\theta\in (0,1)$ such that for $k\in \mathbb{N}$  we have 
\begin{align*}
    \TwoNorm{\nabla u}{\cilt \backslash {\omega}_{\vnode,k}}\lesssim \theta^k \TwoNorm{\nabla u}{\cilt }.
\end{align*}
\end{lemma}
\begin{proof}
Let $\eta_{\vnode}^{k,l}$ be the cut-off function as in the previous lemma for $l<k-1$,
$\tilde{u}=\eta_{\vnode}^{k,l} u -\Qint(\eta_{\vnode}^{k,l} u )\in {V}^f(\cilt\backslash {\omega}_{\vnode,k-l-1})$, 
and note that from Lemma \ref{qi} we have 
\begin{align}\label{qiestimate}
\TwoNorm{\nabla( \eta_{\vnode}^{k,l}u-\tilde{u})}{\cilt }
=\TwoNorm{\nabla \Qint(\eta_{\vnode}^{k,l} u ) }{\cilt }
\lesssim  l^{-1}\TwoNorm{ \nabla u }{ {\omega}_{\vnode,k+1}\backslash {\omega}_{\vnode,k-l-1} }.
\end{align}
From this estimate and the properties of $F$ we have 
\begin{align}\label{relation}
\int_{\cilt\backslash {\omega}_{\vnode,k-l-1}} B(x) \nabla u \nabla \tilde{u}\, y^a \,d\x
=\int_{\cilt}B(x) \nabla u \nabla \tilde{u} \, y^a \,d\x=F(\tilde{u})=0.
\end{align}
We utilize a version of the Caccioppoli inequality 
from Lemma \ref{Caccioppoli} to deduce
\begin{align*}
\TwoNorm{\nabla u}{\cilt\backslash {\omega}_{\vnode,k}}^2
&\lesssim \int_{\cilt\backslash {\omega}_{\vnode,k-l-1}} \eta_{\vnode}^{k,l} B(x)\nabla u \nabla u  \, y^a\,d\x \\
&\lesssim  \int_{\cilt\backslash {\omega}_{\vnode,k-l-1}} \nabla u\left(\nabla (\eta_{\vnode}^{k,l} u ) -u\nabla \eta_{\vnode}^{k,l}\right)\, y^a\,d\x.
\end{align*}	
Using the fact that $\Qint(u)=0$, estimate \eqref{qiestimate}, and  the relation \eqref{relation} we have 
\begin{align*}
&\TwoNorm{\nabla u}{\cilt\backslash {\omega}_{\vnode,k}}^2
\lesssim  \int_{\cilt\backslash {\omega}_{\vnode,k-l-1}} \nabla u(\nabla (\eta_{\vnode}^{k,l} u -\tilde{u}))\, y^a \,d\x
-\int_{\cilt\backslash {\omega}_{\vnode,k-l-1}}\nabla u(u-\Qint(u))\nabla \eta_{\vnode}^{k,l} \, y^a \,d\x\\
&\qquad\lesssim   l^{-1}\TwoNorm{ \nabla u }{ \cilt \backslash {\omega}_{\vnode,k-l-1} }^2
+(l H)^{-1}\TwoNorm{ \nabla u }{ \cilt \backslash {\omega}_{\vnode,k-l-1} }\TwoNorm{  u -\Qint(u)}{ \cilt \backslash {\omega}_{\vnode,k-l-1} }    \\
&\qquad\lesssim l^{-1}  \TwoNorm{ \nabla u }{ \cilt \backslash {\omega}_{\vnode,k-l-1} }^2.
\end{align*}	
On the last term we used the approximation property \eqref{stableproj1}.
Successive applications of 
the above estimate leads to 
\begin{align*}
\TwoNorm{\nabla u}{\cilt\backslash {\omega}_{\vnode,k}}^2&\lesssim  l^{-1} \TwoNorm{ \nabla u }{ \cilt \backslash {\omega}_{\vnode,k-l-1} }^2
\lesssim  l^{- \lfloor \frac{k-1}{l+1}  \rfloor} \TwoNorm{ \nabla u }{ \cilt \backslash {\omega}_{\vnode,1} }^2
\lesssim  l^{- \lfloor \frac{k-1}{l+1}  \rfloor} \TwoNorm{ \nabla u }{ \cilt  }^2.
\end{align*}
Finally, noting that
\[
  \left\lfloor \frac{k-1}{l+1}  \right\rfloor = \left\lceil \frac{k-l-1}{l+1}  \right\rceil \geq \frac{k}{l+1} -1,
\]
taking $\theta=l^{-\frac{1}{l+1}}$ yields the result.
\end{proof}
\par
We are now ready to restate our result on the error introduced from localization. 
This is merely Lemma \ref{localglobal.derp} restated.
When $k$ is sufficiently large so that the corrector problem is all of $\cilt$, 
we denote $Q_{\vnode,k}=Q_{\vnode,\cilt}$.
Let $u_{H}\in {V}_{H}$, let $Q_{k}$ be constructed from \eqref{Qcorrector}, 
and $Q_{\cilt}$ defined to be the ideal corrector without truncation, then 
\begin{align}\label{localglobaleq}
    \TwoNorm{\nabla( Q_{\cilt}(u_{H})-Q_{k}(u_{H}))  }{\cilt}\lesssim  k^{\frac{d}{2}} \theta^{k} \TwoNorm{ \nabla u_{H} }{\cilt}.
\end{align}
We begin the proof in a similar way as in \cite{brown2016multiscale}.
\begin{proof}[Proof of Lemma \ref{localglobal.derp}]
We denote $v= Q_{\cilt}(u_{H})-Q_{k}(u_{H})\in {V}^f,$ subsequently $\Qint(v)=0$.
Taking the cut-off function $\eta_{\vnode}^{k,1}$ we have 
\begin{align}
\label{term1}
\TwoNorm{\nabla v}{\cilt}^2
&\lesssim \sum_{\vnode\in {\cal N}_{dof}}\int_{\cilt}B(x) \nabla( Q_{\vnode,\cilt}(u_{H})-Q_{\vnode,k}(u_{H}))\nabla (v(1-\eta_{\vnode}^{k,1}))\, y^a\,d\x \\
\label{term2}
&+\sum_{\vnode\in {\cal N}_{dof}}\int_{\cilt}B(x) \nabla( Q_{\vnode,\cilt}(u_{H})-Q_{\vnode,k}(u_{H}))\nabla (v\eta_{\vnode}^{k,1})\, y^a \,d\x.
\end{align}
Estimating the right hand side of \eqref{term1} for each $\vnode$,
and using the boundedness of $B(x)$, we have
\begin{align*}
&\int_{\cilt}B(x)\nabla( Q_{\vnode,\cilt}(u_{H})-Q_{\vnode,k}(u_{H}))\nabla (v(1-\eta_{\vnode}^{k,1}))\, y^a \,d\x\\
&\qquad\lesssim \TwoNorm{\nabla( Q_{\vnode,\cilt}(u_{H})-Q_{\vnode,k}(u_{H}))}{\cilt} \TwoNorm{\nabla (v(1-\eta_{\vnode}^{k,1}))}{{\omega}_{\vnode,k}}\\
&\qquad\lesssim  \TwoNorm{\nabla( Q_{\vnode,\cilt}(u_{H})-Q_{\vnode,k}(u_{H}))}{\cilt} \left(\TwoNorm{\nabla v}{{\omega}_{\vnode,k}}
+  \TwoNorm{v\nabla (1-\eta_{\vnode}^{k,1}))}{{\omega}_{\vnode,k}\backslash {\omega}_{\vnode,k-1}}   \right)\\
&\qquad\lesssim \TwoNorm{\nabla( Q_{\vnode,\cilt}(u_{H})-Q_{\vnode,k}(u_{H}))}{\cilt} \left(\TwoNorm{\nabla v}{{\omega}_{\vnode,k}} 
+  H^{-1}\TwoNorm{v-\Qint(v)}{{\omega}_{\vnode,k}\backslash {\omega}_{\vnode,k-1}}   \right)\\
&\qquad\lesssim  \TwoNorm{\nabla( Q_{\vnode,\cilt}(u_{H})-Q_{\vnode,k}(u_{H}))}{\cilt} \TwoNorm{\nabla v}{{\omega}_{\vnode,k+1}}.
\end{align*}
As in the proof of Lemma \ref{decaylemma}, we denote $\tilde{v}=\eta_{\vnode}^{k,1} v -
\Qint(\eta_{\vnode}^{k,1} v )\in {V}^f(\cilt)$ 
and so $\tilde{v}$ satisfies  
\begin{align*}
    \int_{\cilt}B(x) \nabla( Q_{\vnode,\cilt}(u_{H})-Q_{\vnode,k}(u_{H}))\nabla\tilde{v}   \, y^a \,d\x=0.
\end{align*}
We have now the estimate for \eqref{term2} for $\vnode\in {\cal N}_{dof}$ 
using the above identity and \eqref{qiestimate}
\begin{align*}
&\int_{\cilt}B(x) \nabla( Q_{\vnode,\cilt}(u_{H})-Q_{\vnode,k}(u_{H}))\nabla (v\eta_{\vnode}^{k,1}-\tilde{v})\, y^a \,d\x\\
&\qquad\lesssim \TwoNorm{ \nabla( Q_{\vnode,\cilt}(u_{H})-Q_{\vnode,k}(u_{H}))  }{ \cilt   }\TwoNorm{ \nabla (v\eta_{\vnode}^{k,1}-\tilde{v})  }{ \cilt  }\\
&\qquad\lesssim  \TwoNorm{ \nabla( Q_{\vnode,\cilt}(u_{H})-Q_{\vnode,k}(u_{H}))  }{ \cilt  }\TwoNorm{ \nabla v }{ {\omega}_{\vnode ,k+1 } }.
\end{align*}
Combing the estimates for \eqref{term1} and \eqref{term2} we obtain 
\begin{align}
\TwoNorm{\nabla v}{\cilt}^2
&\lesssim \sum_{\vnode\in {\cal N}_{dof}} 
\TwoNorm{\nabla( Q_{\vnode,\cilt}(u_{H})-Q_{\vnode,k}(u_{H}))}{\cilt}  \TwoNorm{\nabla v}{{\omega_{\vnode,k+1}}}\nonumber\\
\label{vestimate}
&\lesssim  k^{\frac{d}{2}} \left(\sum_{\vnode\in {\cal N}_{dof}}
\TwoNorm{ \nabla( Q_{\vnode,\cilt}(u_{H})-Q_{\vnode,k}(u_{H}))  }{ \cilt  }^2\right)^{\frac{1}{2}}\TwoNorm{ \nabla v }{\cilt },
\end{align}
supposing that 
$\# \{\vnode\in {\cal N}_{dof}|{\omega}_{\vnode'}\subset {\omega}_{\vnode,k+1}\}\lesssim k^{d}$,
as is guaranteed by quasi-uniformity of the coarse-grid.
\par
For $\vnode \in{\cal N}_{dof}$, we estimate 
$\TwoNorm{ \nabla( Q_{\vnode,\cilt}(u_{H})-Q_{\vnode,k}(u_{H}))  }{ \cilt  }$ and 
we use the Galerkin orthogonality of the local problem, that is 
\begin{align*}
\TwoNorm{ \nabla( Q_{\vnode,\cilt}(u_{H})-Q_{\vnode,k}(u_{H}))  }{ \cilt  }
\leq \inf_{q_\vnode\in {V}^f({\omega}_{\vnode,k}) } \TwoNorm{ \nabla( Q_{\vnode,\cilt}(u_{H})-q)  }{ \cilt  }.
\end{align*}
Let 
$q_{\vnode}=(1-\eta^{(k-1),1}_{\vnode})Q_{\vnode,\cilt}(u_{H})-\Qint((1-\eta^{(k-1),1}_{\vnode})Q_{\vnode,\cilt}(u_{H})  )\in {V}^f({\omega}_{\vnode,k})$, 
we have
\begin{align*}
&\TwoNorm{ \nabla( Q_{\vnode,\cilt}(u_{H})-Q_{\vnode,k}(u_{H}))  }{ \cilt  }^2\\
&\qquad\leq \TwoNorm{ \nabla( \eta^{(k-1),1}_{\vnode}Q_{\vnode,\cilt}(u_{H})-\Qint((1-\eta^{(k-1),1}_{\vnode})Q_{\vnode,\cilt}(u_{H})  )) }{ \cilt  }^2\\
&\qquad\lesssim \TwoNorm{ \nabla Q_{\vnode,\cilt}(u_{H}) }{\cilt \backslash {\omega}_{\vnode,k-2} }^2
+ \TwoNorm{ \nabla(\Qint(\eta^{(k-1),1}_{\vnode}Q_{\vnode,\cilt}(u_{H})  )) }{ \cilt  }^2.
\end{align*}
Using Lemma \ref{qi} and Lemma \ref{decaylemma} on the second term we arrive at 
\begin{align*}
&\TwoNorm{ \nabla( Q_{\vnode,\cilt}(u_{H})-Q_{\vnode,k}(u_{H}))  }{ \cilt  }^2\\
&\qquad\lesssim \TwoNorm{ \nabla Q_{\vnode,\cilt}(u_{H}) }{ \cilt \backslash {\omega}_{\vnode,k-2} }^2
+ \TwoNorm{ \nabla Q_{\vnode,\cilt}(u_{H})  ) }{ {\omega}_{\vnode,k} \backslash {\omega}_{\vnode,k-3}   }^2\\
&\qquad\lesssim  \TwoNorm{ \nabla Q_{\vnode,\cilt }(u_{H}) }{ \cilt \backslash {\omega}_{\vnode,k-3} }^2\\
&\qquad\lesssim  \theta^{2(k-3) } \TwoNorm{ \nabla Q_{\vnode,\cilt}(u_{H}) }{ \cilt  }^2.
\end{align*} 
From the definition of $Q_{\vnode,\cilt}$ from \eqref{Qcorrector} with global corrector patches,
we get
\begin{align*}
\TwoNorm{ \nabla( Q_{\vnode,\cilt}(u_{H})-Q_{\vnode,k}(u_{H}))  }{ \cilt  }^2
\lesssim \theta^{2k} \TwoNorm{ \nabla u_{H} }{ {\omega}_{\vnode} }^2.
\end{align*}
Thus, summing over all $\vnode\in {\cal N}_{dof}$  and combining the above with \eqref{vestimate} concludes the proof.
\end{proof}

\bibliographystyle{abbrv}
\bibliography{HMP_references.bib}

\end{document}